\documentclass[letterpaper, 11pt,  reqno]{amsart}

\usepackage{amsmath,amssymb,amscd,amsthm,amsxtra, esint}
\usepackage{color}

\usepackage[implicit=true]{hyperref}

\allowdisplaybreaks[2]

\definecolor{gr}{rgb}   {0.,   0.69,   0.23 }
\definecolor{bl}{rgb}   {0.,   0.5,   1. }
\definecolor{mg}{rgb}   {0.85,  0.,    0.85}
\definecolor{yl}{rgb}   {0.8,  0.7,   0.}
\definecolor{or}{rgb}  {0.7,0.2,0.2}

\newcommand{\dis}{\displaystyle}

\newtheorem{theorem}{Theorem} [section]

\newtheorem{lemma}[theorem]{Lemma}
\newtheorem{proposition}[theorem]{Proposition}
\newtheorem{remark}[theorem]{Remark}

\newtheorem{definition}[theorem]{Definition}
\newtheorem{corollary}[theorem]{Corollary}

\newcommand{\I}{\hspace{0.5mm}\text{I}\hspace{0.5mm}}
\newcommand{\II}{\text{I \hspace{-2.8mm} I} }
\newcommand{\III}{\text{I \hspace{-2.9mm} I \hspace{-2.9mm} I}}
\newcommand{\IV}{\text{I \hspace{-2.9mm} V}}
\newcommand{\noi}{\noindent}
\newcommand{\Z}{\mathbb{Z}}
\newcommand{\R}{\mathbb{R}}
\newcommand{\C}{\mathbb{C}}
\newcommand{\T}{\mathbb{T}}

\let\Re=\undefined\DeclareMathOperator*{\Re}{Re}
\let\Im=\undefined\DeclareMathOperator*{\Im}{Im}
\let\P= \undefined
\newcommand{\P}{\mathbf{P}}

\newcommand{\E}{\mathbb{E}}

\renewcommand{\L}{\mathcal{L}}

\newcommand{\F}{\mathcal{F}}

\newcommand{\al}{\alpha}
\newcommand{\be}{\beta}
\newcommand{\dl}{\delta}

\newcommand{\nb}{\nabla}

\newcommand{\Dl}{\Delta}
\newcommand{\eps}{\varepsilon}

\newcommand{\g}{\gamma}
\newcommand{\G}{\Gamma}
\newcommand{\ld}{\lambda}

\newcommand{\s}{\sigma}
\newcommand{\Si}{\Sigma}
\newcommand{\ft}{\widehat}

\newcommand{\wt}{\widetilde}
\newcommand{\cj}{\overline}
\newcommand{\dx}{\partial_x}
\newcommand{\dt}{\partial_t}
\newcommand{\dd}{\partial}

\renewcommand{\l}{\ell}
\renewcommand{\o}{\omega}
\renewcommand{\O}{\Omega}

\newcommand{\les}{\lesssim}
\newcommand{\ges}{\gtrsim}

\newcommand{\jb}[1]
{\langle #1 \rangle}

\newcommand{\ind}{\mathbf 1}

\newcommand{\M}{\mathcal{M}}

\newcommand{\N}{\mathbb{N}}

\newcommand{\Pk}{P^{(2m)}_2}

\newcommand{\Pkn}{P^{(2m)}_{2, N}}
\renewcommand{\H}{\mathcal{H}}
\newcommand{\Li}{L^{(1)}}

\newtheorem*{ackno}{Acknowledgements}

\numberwithin{equation}{section}
\numberwithin{theorem}{section}

\begin{document}
\baselineskip = 15pt

\title[Invariant Gibbs measures for the 2-$d$ defocusing NLS]
{A pedestrian approach to  the invariant Gibbs measures for 
the 2-$d$ defocusing nonlinear Schr\"odinger equations}

\author[T.~Oh and  L.~Thomann]
{Tadahiro Oh and Laurent Thomann}

\address{
Tadahiro Oh, School of Mathematics\\
The University of Edinburgh\\
and The Maxwell Institute for the Mathematical Sciences\\
James Clerk Maxwell Building\\
The King's Buildings\\
Peter Guthrie Tait Road\\
Edinburgh\\ 
EH9 3FD\\
 United Kingdom}

\email{hiro.oh@ed.ac.uk}

\address{
Laurent Thomann\\
Institut  \'Elie Cartan, Universit\'e de Lorraine, B.P. 70239,
F-54506 Vandoeuvre-l\`es-Nancy Cedex, France}

\email{laurent.thomann@univ-lorraine.fr}



\begin{abstract}
We consider the defocusing nonlinear Schr\"odinger equations on 
the two-dimensional compact Riemannian manifold without boundary or a bounded domain in $\R^2$.
Our aim is to give a pedagogic and self-contained presentation on 
the Wick renormalization
in terms of the Hermite polynomials and the Laguerre polynomials
and construct the Gibbs measures corresponding to the Wick ordered Hamiltonian.
Then, we construct global-in-time solutions
with initial data distributed according to the Gibbs measure
and show that the law of the random solutions, at any time,
is again given by the Gibbs measure.

\end{abstract}

%
\maketitle
\tableofcontents

\baselineskip = 14pt

\section{Introduction}

\subsection{Nonlinear Schr\"odinger equations}

 Let $(\M,g)$ be  a two-dimensional compact Riemannian manifold without boundary or a bounded domain in $\R^2$.
We consider the defocusing nonlinear Schr\"odinger equation (NLS): 
\begin{align}\label{MNLS}
\begin{cases}
i \dt u + \Dl_g  u = |u|^{k-2}u \\
u|_{t = 0} = \phi,
\end{cases}
\qquad  (t,x) \in \R \times \M,
\end{align}
where  $\Delta_g$ stands for the Laplace-Beltrami operator on $\M$, $k = 2m \geq 4$ is an even integer, and  
 the unknown is the function $u: \R \times \M \longrightarrow \C$. 
 
\medskip

The aim of this article is to give a pedagogic and self-contained\footnote{with the exception 
of the Wiener chaos estimate (Lemma \ref{LEM:hyp3}).}
presentation on the construction of  an invariant  Gibbs measure for a renormalized version of~\eqref{MNLS}.  
In particular, we present an elementary Fourier analytic approach to the problem
in the hope that this will be accessible to readers (in particular those in dispersive PDEs)
without prior knowledge in quantum field theory and/or stochastic analysis.
In order to make the presentation simpler, we first detail the case of the flat torus 
$\M=\T^{2}$, where  $\T = \R/(2\pi \Z)$. 
Namely, we consider 
\begin{align}
\begin{cases}
i \dt u + \Dl  u = |u|^{k-2}u \\
u|_{t = 0} =  \phi,
\end{cases}
\qquad  (t,x) \in \R \times \T^2.
\label{NLS1}
\end{align}

The equation~\eqref{NLS1} is known to possess 
the following Hamiltonian structure:
\begin{equation}
 \dt u =  - i  \frac{\dd H}{\dd \cj u}, 
\label{Hamil0}
\end{equation}

\noi
where $H = H(u)$ is  the Hamiltonian given by 
\begin{align}
H(u) = \frac{1}{2}\int_{\T^2} |\nb u|^2dx
+ \frac1{k} \int_{\T^2} |u|^k dx.
\label{Hamil}
\end{align}

\noi
Moreover, 
 the mass 
 \[ M(u) = \int_{\T^2} |u|^2 dx\]
is also conserved under the dynamics of 
\eqref{NLS1}.

\begin{samepage}

%
\subsection{Gibbs measures}

Given a Hamiltonian flow on $\R^{2n}$:
\begin{equation} \label{HR2}
\begin{cases}
\dot{p}_j = \frac{\partial H}{\partial q_j} \\
\dot{q}_j = - \frac{\partial H}{\partial p_j}
\end{cases}
\end{equation}

\end{samepage}

\noi
 with Hamiltonian $ H (p, q)= H(p_1, \cdots, p_n, q_1, \cdots, q_n)$,
Liouville's theorem states that the Lebesgue measure
$\prod_{j = 1}^n dp_j dq_j$
on $\mathbb{R}^{2n}$ is invariant under the flow.
Then, it follows from the conservation of the Hamiltonian $H$
that  the Gibbs measures
$e^{-\beta H(p, q)} \prod_{j = 1}^{n} dp_j dq_j$ are invariant under the dynamics of~\eqref{HR2}.
Here, $\beta> 0$ denotes the reciprocal temperature.

NLS~\eqref{NLS1} is a Hamiltonian PDE,
where the Hamiltonian is conserved under its dynamics.
Thus by drawing an analogy to the finite dimensional setting, one expects the Gibbs measure of the form:\footnote{In the following, 
$Z$,  $Z_N$, and etc.~denote various  normalizing constants
so that the corresponding measures are probability measures when appropriate.}
\begin{align}
 \text{``}d\Pk = Z^{-1} \exp(- \be H(u))du\text{''}
\label{G1}
 \end{align}

\noi
to be invariant under the dynamics of~\eqref{NLS1}.\footnote{For simplicity, 
we set $\be = 1$ in the following.
See~\cite{OQ} for a discussion on the Gibbs measures
and different values of $\be > 0$.}
As it is,~\eqref{G1} is merely a formal expression
and we need to give a precise meaning.
From~\eqref{Hamil}, we can write~\eqref{G1} as
\begin{align}
\text{``}d\Pk = Z^{-1} e^{-\frac 1{2m} \int |u|^{2m} dx} e^{-\frac 12 \int |\nb u |^2 dx} du\text{''}.
\label{G2}
 \end{align}

\noi
This motivates us to define the Gibbs measure $\Pk$
as an absolutely continuous (probability) measure with respect to the 
following 
massless Gaussian free field:
$ d\mu = \wt Z^{-1} \exp\big(-\frac 12 \int |\nb u |^2 dx\big) du.$
In order to avoid the problem at the zeroth frequency, 
we instead consider 
the following 
massive Gaussian free field:
\begin{align}
 d\mu = \wt Z^{-1} e^{-\frac 12 \int |\nb u |^2 dx - \frac{1}{2}\int|u|^2 dx } du.
\label{G3}
\end{align}

\noi
in the following.
Note that this additional factor 
replaces $-H(u)$  by $-H(u) - \frac 12 M(u)$
in the formal definition~\eqref{G1} of $\Pk$.
In view of the conservation of mass, 
however, 
we still expect $\Pk$ to be invariant
if we can give a proper meaning to $\Pk$.

It is well known that 
 $\mu$ in~\eqref{G3}
corresponds to a mean-zero Gaussian  free  field on $\T^2$.
More precisely, $\mu$ is the mean-zero Gaussian measure
on $H^s(\T^2)$ for any $ s< 0$
with the covariance operator $Q_s = (\text{Id}-\Dl)^{-1+s}$.
Recall that a covariance operator $Q$
of a mean-zero probability measure $\mu$ on a Hilbert space $\H$
is a trace class operator, satisfying
\begin{align}
 \int_{\H} \jb{f, u}_\H \cj{\jb{h, u}}_\H d\mu (u)
= \jb{Q f, h}_\H
\label{G4}
\end{align}

\noi
for all $f, h \in \H$.

We can also view the Gaussian measure $\mu$
as the induced probability measure
under the map:\footnote{Strictly speaking, 
there is a factor of $(2\pi)^{-1}$ in~\eqref{G5}.
For simplicity of the presentation, however, 
we drop such harmless $2\pi$ hereafter.}
\begin{align}
 \o \in \O \longmapsto u(x) = u(x; \o) = \sum_{n \in \Z^2} \frac{g_n(\o)}{\sqrt{1+|n|^2}}e^{in\cdot x},
\label{G5}
 \end{align}

\noi
where $\{g_n\}_{n \in \Z^2}$
is a sequence of independent standard\footnote{Namely, $g_n$ has mean 0
and $\text{Var}(g_n) = 1$.} complex-valued Gaussian
random variables on a probability space $(\O, \F, P)$.
Namely, functions under $\mu$ are represented by 
the random Fourier series given in~\eqref{G5}.
Note that the random function $u$
in~\eqref{G5}
is in $H^{s}(\T^2)\setminus L^2(\T^2)$ for any $s < 0$, almost surely.
Thus, 
$\mu$ is a Gaussian probability measure on 
$H^s(\T^2)$ for any $s< 0$.
Moreover,  it is easy to see that 
\eqref{G4} with $\H = H^s(\T^2)$ $Q_s = (\text{Id}-\Dl)^{-1+s}$, $s< 0$, 
follows from~\eqref{G5}.  Indeed, we have
\begin{align}
 \int_{H^s} \jb{f, u}_{H^s} \cj{\jb{h, u}}_{H^s} d\mu (u)
& =  \E\bigg[ 
 \sum_{n \in \Z^2} \frac{\ft f(n) \cj{g_n(\o)}}{\jb{n}^{1-2s}}
 \sum_{m \in \Z^2} \frac{\cj{\ft h(m)} g_m(\o)}{\jb{m}^{1-2s}}
 \bigg]\notag \\
& = 
 \sum_{n \in \Z^2} \frac{\ft f(n)\cj{\ft h(n)}}{\jb{n}^{2-4s}}
= \jb{Q_s f, h}_{H^s}.
\label{G6}
\end{align}

\noi
Here, $\jb{\,\cdot\, } = (1 + |\cdot|)^\frac{1}{2}$.
Note that the second equality in~\eqref{G6} holds even for $s\geq0$.
For $s\geq 0$, however, $\mu$ is not a probability measure on $H^s(\T^2)$.
Indeed, we have  $\mu(L^2(\T^2)) = 0$.

The next step is to make sense of the Gibbs measure $\Pk$ in~\eqref{G2}.
First, let us briefly go over the situation when $d = 1$.
In this case, $\mu$ defined by~\eqref{G3}
is a probability measure on $H^s(\T)$, $s < \frac{1}{2}$.
Then, it follows from Sobolev's inequality
that $\int_\T|u(x; \o)|^k dx$ is finite almost surely.
Hence, for any $k > 2$, the Gibbs measure:
\begin{align}
d P^{(k)}_1 = Z^{-1} e^{-\frac 1k \int_{\T} |u|^k dx} d\mu 
\label{G7}
 \end{align}

\noi
is a probability measure 
on $H^s(\T)$, $s < \frac{1}{2}$,
absolutely continuous with respect to $\mu$.
Moreover, by constructing global-in-time dynamics
in the support of $ P^{(k)}_1$, 
Bourgain~\cite{BO94} proved that 
the Gibbs measure
$ P^{(k)}_1$ is invariant under the dynamics of 
the defocusing NLS for $k > 2$. 
Here, by invariance, we mean that 
\begin{align}
P^{(k)}_1\big(\Phi(-t) A\big) = P^{(k)}_1(A)
\label{inv1}
\end{align}

\noi
for any measurable set $A \in \mathcal{B}_{H^s(\T)}$
and any $t \in \R$, 
where $\Phi(t): u_0 \in H^s(\T) \mapsto 
u(t) = \Phi(t) u_0 \in H^s(\T)$
is a well-defined solution map, at least almost surely with respect to 
$P^{(k)}_1$.
McKean~\cite{McKean} gave 
an independent proof 
of the invariance of the Gibbs measure when $k = 4$,
relying on a probabilistic argument.
See Remark~\ref{REM:focusing} below
for the discussion on the focusing case.
Over the recent years, 
there has been a significant progress 
in the study of invariant Gibbs measures
for Hamiltonian PDEs.
See, for example, \cite{LRS, BO94, MV, McKean, BO96, BO97,  Zhid2,
TZ1, TZ2, BTIMRN, BT2, OH3,  OHSBO, Tzv, TTz, NORS, OQV, 
 SuzzoniBBM, BTT,  DengBO,  R, 
BTT1}.

The situation for $d = 2$ is entirely different.
As discussed above, the random function $u$ in~\eqref{G5}
is not in $L^2(\T^{2})$ almost surely.
This in particular implies that 
\begin{align}
 \int_{\T^2} |u(x; \o)|^k dx = \infty
\label{G8}
 \end{align}

\noi
almost surely for any $k \geq 2$.
Therefore, we can not construct a probability measure of the form:
\begin{align}
d P_2^{(k)} = Z^{-1} e^{-\frac 1k \int_{\T^2} |u|^k dx} d\mu. 
\label{G9}
 \end{align}

\noi
Thus, we are required to perform a (Wick) renormalization 
on the nonlinear part $|u|^{k}$
of the Hamiltonian.
This is a well studied subject 
in the Euclidean quantum field theory, at least in the real-valued setting.
See Simon~\cite{Simon} and  Glimm-Jaffe~\cite{GJ}.
Also, see Da Prato-Tubaro~\cite{DPT1}
for a concise discussion on $\T^2$,
where the Gibbs measures naturally appear in the context of the stochastic
quantization equation.

\subsection{Wick renormalization} 

There are different ways
to  introduce the Wick renormalization. 
One classical way is to  use the Fock-space formalism, 
where the Wick ordering is given as the reordering  of the creation and annihilation operators.  
See \cite{Simon, Meyer, DerezinskiGerard} for more details. 
It can be also defined through the multiple Wiener-Ito integrals.
In the following, we directly define it as the orthogonal projection onto 
the Wiener homogeneous chaoses  (see the Wiener-Ito decomposition \eqref{IW} below)
by using the Hermite polynomials and the (generalized) Laguerre polynomials,
since this allows us to introduce only  the necessary objects
without introducing cumbersome notations and formalism,
making our presentation accessible to readers without prior knowledge in the problem.

Before we study the Wick renormalization for NLS, let us briefly discuss the Wick renormalization on $\T^2$ in the real-valued setting.
We refer to  \cite{DPT1} 
for more details.
 We assume that $u$ is real-valued.
 Then,  the random function $u$ under $\mu$ in~\eqref{G3}
is represented by the random Fourier series~\eqref{G5}
conditioned that $g_{-n} = \cj{g_n}$.
Given $N \in \N$, 
let $\P_N$ be the Dirichlet projection onto the frequencies $\{|n|\leq N\}$
and set $u_N = \P_N u$, where $u$ is as in~\eqref{G5}.
Note that, for each $x \in \T^2$,  
the random variable $u_N(x)$ is a mean-zero real-valued Gaussian with variance
\begin{align}
\s_N: = \E [u_N^2(x)]
= \sum_{|n|\leq N}\frac{1}{1+|n|^2} \sim \log N.
\label{Wick1a}
\end{align}

\noi
Note that $\s_N$ is independent of $x \in \T^2$.
Fix  an even integer
$k \geq 4$. We define the Wick ordered monomial $:\! u_N^k  \!:$
by 
\begin{align}
:\! u_N^k \!: \, = H_k(u_N; \s_N), 
\label{Wick1}
\end{align}
	
\noi
where $H_k(x; \s)$ is the Hermite polynomial of degree $k$
defined in~\eqref{H1}.
Then, one can show that the limit
\begin{align}
\int_{\T^2} :\! u^k\!: dx
= \lim_{N \to \infty} 
\int_{\T^2}:\! u_N^k\!: dx
\label{Wick2}
\end{align}
	
\noi
exists in $L^p( \mu)$
for any finite $p \geq 1$.
Moreover, one can construct the Gibbs measure:
\begin{align*}
d P_2^{(k)} = Z^{-1} e^{-\frac 1k \int_{\T^2} :u^k : \,dx} d\mu
 \end{align*}

\noi
as the  limit of 
\begin{align*}
d P_{2,N}^{(k)} = Z_N^{-1} e^{-\frac 1k \int_{\T^2} :u_N^k : \, dx} d\mu.
 \end{align*}

\noi
The key ingredients of the proof of the above claims are
the Wiener-Ito decomposition of $L^2(H^s(\T^2), \mu)$ for $s< 0$, 
the hypercontractivity of the Ornstein-Uhlenbeck semigroup, 
and
Nelson's estimate~\cite{Nelson2, Nelson}.\medskip

For our problem on 
NLS~\eqref{NLS1}, we need to work on complex-valued functions.
In the real-valued setting, 
the Wick ordering was defined by 
the Hermite polynomials.
In the complex-valued setting, we also define the Wick ordering
by the Hermite polynomials, but through  applying the Wick
ordering the real and imaginary parts separately.

 Let   $ u$ be as in \eqref{G5}.
Given  
 $N\in \N$, we define $u_N$ by
\[u_N = \P_N u= \sum_{|n|\leq N }\ft u (n)e^{in\cdot x}, \]

\noi
where $\P_N$ is the Dirichlet projection onto the frequencies $\{|n|\leq N\}$ as above.
Then, 
for $m \in \N$, 
we define the Wick ordered 
monomial $:\! |u_N|^{2m}  \!:$
by 
\begin{align}
:\! |u_N |^{2m} \!: \, 
&
= \, :\! \big((\Re u_N)^2 + (\Im u_N)^2\big)^m\!: \,  \notag \\
& = \sum_{\l = 0}^m 
\begin{pmatrix}
m\\ \l
\end{pmatrix}
  \, :\! (\Re u_N)^{2\l} \!:    \, :\! (\Im u_N)^{2(m - \l)}\!:.
 \label{L0}
\end{align}

\noi
It turns out, however, that it is more convenient 
to work with 
the Laguerre polynomials in the current complex-valued  setting; see Section~\ref{SEC:2}.
Recall that the Laguerre polynomials $L_m(x)$
are 
defined through the following generating function:
\begin{equation}
G(t, x) : = \frac{1}{1-t} e^{- \frac{tx}{1-t}} = \sum_{m = 0}^\infty t^mL_m(x),
\label{L1}
 \end{equation}
	
\noi
for $|t| < 1$ and  $x\in \R$.
For readers' convenience, we write out the first few Laguerre polynomials
in the following:
\begin{align}
& 
\hphantom{XXXXX}
L_0(x) = 1, 
\qquad 
L_1(x) = -x+1, 
\qquad
L_2(x) = \tfrac{1}{2}(x^2 -4x + 2), \notag \\
&
 L_3(x) = \tfrac{1}{3!}(-x^3 + 9 x^2 - 18 x +6), 
\qquad 
L_4(x) = \tfrac{1}{4!}(x^4 - 16x^3 +72x^2 -96x +24).
\label{L1a}
\end{align}
More generally, the $L_{m}$ are given by the formula
\begin{align}
L_m(x) = 
\sum_{\l = 0}^m
\begin{pmatrix}
m\\ \l
\end{pmatrix}
\frac{(-1)^\l}{\l!} x^\l.
\label{L1b}
\end{align}

\noi
Given $\s >0$, we set 
\begin{align}
 L_m(x; \s) := \s^m L_m\big(\tfrac{x}{\s}\big).
\label{L2}
 \end{align}

\noi
Note that $L_m(x; \s)$ is a homogenous polynomial  of degree $m$ in $x$ and $\s$.
Then, given  $N \in \N$, 
we can rewrite 
the Wick ordered 
monomial $:\! |u_N|^{2m}  \!:$ defined in~\eqref{L0} as 
\begin{align}
 \,  :\! |u_N|^{2m} \!: \, =
 (-1)^m m! \cdot L_m(|u_N|^2; \s_N),
\label{L3}
\end{align}

\noi
where $\s_N$ is given by
\begin{align}
\s_N = \E [|u_N(x)|^2]
= \sum_{|n|\leq N}\frac{1}{1+|n|^2} \sim \log N,
\label{Wick4}
\end{align}

\noi
independently of $x \in \T^2$.
See Lemma~\ref{LEM:Wick1}
for the equivalence of~\eqref{L0} and~\eqref{L3}.

For $N \in \N$, let 
\begin{align}
G_N(u) =  \frac 1{2m} \int_{\T^2}:\! |\P_N u|^{2m}\!: dx.
\label{L4}
\end{align}

\noi
Then, we have the following proposition.

\begin{proposition}\label{PROP:WH1}
Let $m \geq2$ be an   integer.
Then, $\{G_N(u)\}_{N\in \N}$
is a Cauchy sequence
in $L^p( \mu)$ for any $p\geq 1$.
More precisely, 
there exists $C_{m} > 0$ such that 
\begin{align*}
\| G_M(u) - G_N(u) \|_{L^p(\mu)}
\leq C_{m} (p-1)^m\frac{1}{N^\frac{1}{2}}
\end{align*}

\noi
for 
 any $p \geq 1$ and 
any $ M \geq N \geq 1$.

\end{proposition}

Proposition~\ref{PROP:WH1} states that we can define the limit $G(u)$ as
\begin{align*}
G(u)=\frac 1{2m} \int_{\T^2} :\! |u|^{2m}\!: dx
= \lim_{N \to \infty} G_N(u)
= \frac 1{2m}
\lim_{N \to \infty} 
\int_{\T^2}:\! |\P_N u|^{2m}\!: dx
\end{align*}

\noi
and that  $G(u) \in L^p( \mu)$
for any finite $p \geq  2$.
This allows us to define  the Wick ordered Hamiltonian:
\begin{align}
H_\text{\tiny Wick}(u) = \frac{1}{2}\int_{\T^2} |\nb u|^2dx
+ \frac1{2m} \int_{\T^2} :\! |u|^{2m}\!: dx
\label{Hamil2}
\end{align}

\noi
for an  integer $m \geq 2$.
In order to discuss the invariance property of the Gibbs measures, 
we need to  overcome the following two problems.
\begin{itemize}
\item[(i)] Define 
 the Gibbs measure of the form 
\begin{align}
\text{``}d \Pk = Z^{-1} e^{-H_{\tiny \text{Wick}}(u) - \frac 12 M(u)} du\text{''} , 
\label{Gibbs1}
 \end{align}

\noi
corresponding to the Wick ordered Hamiltonian $H_\text{\tiny Wick}$.

\item[(ii)] 
Make sense of 
the following defocusing Wick ordered NLS  on $\T^2$:
\begin{align}
i \dt u + \Dl  u = \, :\!|u|^{2(m-1)} u\!: \;, 
\qquad  (t,x) \in \R \times \T^2,
\label{NLS2}
\end{align}

\noi
arising as a Hamiltonian PDE: $\dis \dt u = -i \dd_{\cj u} H_\text{Wick}$. 
In particular, we need to give a precise meaning to 
the Wick ordered nonlinearity 
$:\!|u|^{2(m-1)} u\!:$.
\end{itemize}

Let us first discuss Part (i).
For $N \in \N$, 
let
\begin{equation*}
R_N(u) =  e^{-G_{N}(u)} = e^{-\frac 1{2m} \int_{\T^2} :|u_N|^{2m} : \, dx}
\end{equation*}

\noi
and define 
the truncated Gibbs measure $\Pkn $ by 
\begin{align}
d \Pkn := Z_N^{-1} R_N(u) d\mu = Z_N^{-1} e^{-\frac 1{2m} \int_{\T^2} :| u_N|^{2m} : \, dx} d\mu,
\label{L6}
 \end{align}

\noi
corresponding to the truncated Wick ordered Hamiltonian:
\begin{align}
H^N_\text{\tiny Wick}(u) = \frac{1}{2}\int_{\T^2} |\nb u|^2dx
+ \frac1{2m} \int_{\T^2} :\! |  u_N|^{2m}\!: dx.
\label{Hamil3}
\end{align}

\noi
Note that 
$\Pkn$ is  absolutely continuous with respect to the Gaussian free field $\mu$.

We have the following proposition
on the construction of the Gibbs measure $\Pk$ 
as a limit of $\Pkn$.

\begin{proposition}\label{PROP:Gibbs1}
Let $m \geq 2$ be an  integer.
Then, 
$R_N(u) \in L^p(\mu)$ for any $p\geq 1$
with a uniform bound in $N$, depending on $p \geq 1$.
Moreover, for any finite $p \geq 1$, 
$R_N(u)$ converges to some $R(u)$ in $L^p(\mu)$ as $N \to \infty$.
\end{proposition}

In particular,  
by writing the limit $R(u)\in L^p(\mu)$ as 
\[ R(u) =  e^{-\frac 1{2m} \int_{\T^2} :|u|^{2m} : \, dx},\]

\noi
Proposition~\ref{PROP:Gibbs1}
allows us to define the Gibbs measure $\Pk$  in~\eqref{Gibbs1}  by 
\begin{align}
d \Pk =Z^{-1} R(u) d\mu =  Z^{-1}  e^{-\frac 1{2m} \int_{\T^2} :|u|^{2m} : \,dx} d\mu.
\label{Gibbs2}
 \end{align}

\noi
Then,  $\Pk$  is a probability measure on $H^s(\T^2)$, $s< 0$,
absolutely continuous to the Gaussian field $\mu$.
Moreover, $\Pkn$ converges weakly to $\Pk$.

\medskip


\subsection{Invariant dynamics for the Wick ordered NLS}

In this subsection, we study the dynamical problem
\eqref{NLS2}.
First, we consider the Hamiltonian PDE corresponding to 
the truncated Wick ordered Hamiltonian
$H^N_\text{\tiny Wick}$ in~\eqref{Hamil3}:
\begin{align}
i \dt u^N + \Dl  u^N = \P_N\big(\! :\!|\P_Nu^N|^{2(m-1)} \P_Nu^N\!: \!\big).
\label{NLS3}
\end{align}

\noi
The high frequency part $\P_N^\perp u^N$
evolves according to the linear flow,
while the low frequency part $\P_N u^N$
evolves according to the finite dimensional system of
ODEs viewed on the Fourier side. 
Here,  $\P_N^\perp $ is the Dirichlet projection onto the high frequencies $\{|n|> N\}$.

Let $\mu = \mu_N \otimes \mu_N^\perp$, 
where $\mu_N$  
and $\mu_N^\perp$ are
the marginals of $\mu$ on $E_N = \text{span}\{ e^{in\cdot x}\}_{|n|\leq N}$
and 
$E_N^\perp = \text{span}\{ e^{in\cdot x}\}_{|n|> N}$, respectively.
Then,  we can write $\Pkn$ in~\eqref{L6} as 
\begin{align}
\Pkn = \ft P^{(2m)}_{2, N}\otimes \mu^\perp_N,
\label{NGibbs1}
\end{align} 

\noi
where 
$\ft P^{(2m)}_{2, N}$ is the finite dimensional Gibbs measure defined by 
\begin{align}
d\ft P^{(2m)}_{2, N} = \ft Z_N^{-1} e^{-\frac 1{2m} \int_{\T^2} \, :|\P_N u^N|^{2m} : \, dx} d\mu_N.
\label{NGibbs2} 
\end{align}

\noi
Then, it is easy to see that $\Pkn$
is invariant under the dynamics of~\eqref{NLS3};
see  Lemma~\ref{LEM:global} below.
In particular, 
 the law of $u^N(t)$ is given by $\Pkn$
for any $t \in \R$.

For $N \in \N$, define  $F_N (u)$ by 
\begin{equation}
 F_N(u) = \P_N\big(\! :\!|\P_Nu|^{{2(m-1)}} \P_Nu\!: \!\big).
\label{nonlin2}
 \end{equation}

\noi
Then, assuming that $u$ is distributed according to the Gaussian free field $\mu$
in~\eqref{G3}, 
the following proposition lets us make sense of the Wick ordered
nonlinearity $:\!|u|^{{2(m-1)}} u\!:  $ in~\eqref{NLS2}
as the limit of $F_N(u)$.

\begin{proposition}
\label{PROP:nonlin}
Let $m \geq2$ be an   integer and $s< 0$.
Then, $\{F_N(u)\}_{N\in \N}$
is a Cauchy sequence
in $L^p(\mu;  H^s(\T^2))$ for any $p\geq 1$.
More precisely, 
 given $\eps >0 $ with $s + \eps < 0$, 
there exists $C_{m, s, \eps} > 0$ such that 
\begin{align}
\big\|\| F_M(u) - F_N(u) \|_{H^s}\big\|_{L^p(\mu)}
\leq C_{m, s, \eps} (p-1)^{m- \frac{1}{2}}\frac{1}{N^\eps}
\label{nonlin1}
\end{align}

\noi
for 
 any $p \geq 1$ and 
any $ M \geq N \geq 1$.

\end{proposition}

\noi
In the real-valued setting, 
the nonlinearity corresponding to the Wick ordered Hamiltonian
is again given by a Hermite polynomial.
Indeed, from~\eqref{Wick1}, we have
\[
\tfrac{1}{k} \dd_{u_N}\big(:\! u_N^k \!: \big)=
\tfrac{1}{k}\dd_{u_N}
 H_k(u_N; \s_N)
 = H_{k-1}(u_N; \s_N),
\]

\noi
since $\dx H_k(x; \rho) = k H_{k-1}(x; \rho)$; see~\eqref{H1b}.
The situation is slightly different  in the complex-valued setting.
In the proof of Proposition~\ref{PROP:nonlin}, 
the generalized Laguerre polynomials $L^{(\al)}_m (x)$
with $\al = 1$ plays an important role.
See Section~\ref{SEC:3}.

We denote the limit by $F(u) = \,\, :\! \!|u|^{2(m-1)} u\!:  $
and consider the Wick ordered NLS~\eqref{NLS2}.
When $m= 2$, 
Bourgain~\cite{BO96} 
constructed 
almost sure global-in-time strong solutions
and proved the invariance of the Gibbs measure $P^{(4)}_2$
for the defocusing cubic Wick ordered NLS. See Remark~\ref{REM:NLS} below.
The main novelty in~\cite{BO96} was to construct local-in-time dynamics
in a probabilistic manner, exploiting the gain of integrability
for the random rough linear solution.
By a similar approach,  
Burq-Tzvetkov~\cite{BT1, BT2}
constructed
almost sure global-in-time strong solutions
and proved the invariance of the Gibbs measure
for the defocusing 
subquintic nonlinear wave equation (NLW)
posed on the three-dimensional ball
in the radial setting.

On the one hand, when $m=2$,  
there is only an $\eps$-gap between the regularity of 
the support $H^{s}(\T^2)$, $s< 0$,  of the Gibbs measure
$P^{(4)}_2$ and the scaling criticality $s = 0$
(and the regularity $s>0$ of the known deterministic local well-posedness~\cite{BO93}).
On the other hand, when $m \geq 3$, 
the gap between the regularity of the Gibbs measure $\Pk$
and the scaling criticality 
is slightly more than  $1 - \frac{1}{m-1} \geq \frac 12$. 
At present, it seems very difficult to close this gap
and to construct strong solutions
even in a probabilistic setting.

In the following, 
we instead follow the approach presented
in the work~\cite{BTT1} by the second author with Burq and Tzvetkov.
This work, in turn, was motivated
by the works of Albeverio-Cruzeiro~\cite{AC}
and Da Prato-Debussche~\cite{DPD}
in the study of fluids.
The main idea is to exploit the invariance
of the truncated Gibbs measures $\Pkn$
for~\eqref{NLS3}, 
then to construct 
global-in-time {\it weak} solutions for 
the Wick ordered NLS~\eqref{NLS2},
and finally to prove the invariance
of the Gibbs measure $\Pk$ in some mild sense.

Now, we are ready to state our main theorem.

\begin{theorem}\label{THM:1}
Let $m \geq 2$ be an   integer.
Then, there exists a set $\Si$ of full measure
with respect to $\Pk$
such that for every $\phi \in \Si$, 
the Wick ordered NLS~\eqref{NLS2} with initial condition $u(0) = \phi$
has a global-in-time solution 
\[ u \in C(\R; H^s(\T^2))\]

\noi
for any $s < 0$.
Moreover, for all $t \in \R$, 
the law of the random function $u(t)$ is given by~$\Pk$.
\end{theorem}

There are two components in Theorem~\ref{THM:1}:
existence of solutions and invariance of $\Pk$.
A precursor to the existence part of Theorem~\ref{THM:1}
appears in~\cite{BTT0}.
In~\cite{BTT0}, 
 the second author with Burq and Tzvetkov
used the energy conservation and a regularization 
property under randomization
to construct global-in-time solutions
to the cubic NLW on $\T^d$ for $d \geq 3$.
The main ingredient in~\cite{BTT0} is the compactness of
the solutions to the approximating PDEs.
In order to prove Theorem~\ref{THM:1}, 
we instead follow the argument in~\cite{BTT1}.
Here, the main ingredient
is the tightness (= compactness)
of measures on space-time functions, emanating
from the truncated Gibbs measure $\Pkn $
and Skorokhod's theorem (see Lemma~\ref{LEM:Sk} below).
We point out that 
Theorem~\ref{THM:1}
states only the existence of a global-in-time solution $u$ 
without uniqueness.

Theorem~\ref{THM:1} only claims
that the law $\L(u(t))$ of the $H^s$-valued random variable $u(t)$
satisfies
\[ \L(u(t)) = \Pk\]

\noi
for any $t \in \R$.
This implies the invariance property of the Gibbs measure $\Pk$
in some mild sense, but 
it is weaker than the actual invariance in the sense of~\eqref{inv1}.

\medskip

In fact, 
the result of Theorem~\ref{THM:1} remains true in a more general setting. 
Let~$(\M,g)$ be  a two-dimensional compact Riemannian manifold without boundary or a bounded domain in $\R^2$.
We consider the equation~\eqref{MNLS} on $\M$ (when $\M$ is a domain in $\R^2$, we impose the Dirichlet or Neumann boundary condition). 
Assume that  $k=2m$ for some integer $m\geq 2$. 
  In Section~\ref{SEC:mfd}, 
we prove the analogues
of Propositions~\ref{PROP:WH1},~\ref{PROP:Gibbs1},  and~\ref{PROP:nonlin} 
in this geometric setting, 
by incorporating the geometric information
such as the eigenfunction estimates.
In particular, it is worthwhile to note that 
the variance parameter $\s_N$ in \eqref{Wick4}
now depends on $x \in \M$ in this geometric setting
and more care is needed.
Once we establish 
 the analogues
of Propositions~\ref{PROP:WH1},~\ref{PROP:Gibbs1},  and~\ref{PROP:nonlin},
we can proceed as in the flat torus case. 
Namely, these propositions
 allow us to 
define a renormalized Hamiltonian:
\begin{equation*}
H_\text{\tiny Wick}(u) = \frac{1}{2}\int_{\M} |\nb u|^2dx
+ \frac1{2m} \int_{\M} :\! |u|^{2m}\!: dx,
\end{equation*}

\noi
and a Gibbs measure $\Pk$ as in~\eqref{Gibbs1}. Moreover, we are able to give a sense to NLS with a Wick ordered nonlinearity:
\begin{align}
\begin{cases}
i \dt u + \Dl_g  u = \,\,  : \!|u|^{2(m-1)}u \! : \\
u|_{t = 0} = \phi,
\end{cases}
\qquad  (t,x) \in \R \times \M.
\label{gNLS2}
\end{align}
\medskip 

%
%

%
%

In this general setting,  we have the following result.

\begin{theorem}\label{THM:2}
Let $m \geq 2$ be an   integer.
Then, there exists a set $\Si$ of full measure
with respect to $\Pk$
such that for every $\phi \in \Si$, 
the Wick ordered NLS~\eqref{gNLS2} with initial condition $u(0) = \phi$
has a global-in-time solution 
\[ u \in C(\R; H^s(\M))\]

\noi
for any $s < 0$.
Moreover, for all $t \in \R$, 
the law of the random function $u(t)$ is given by~$\Pk$.
\end{theorem}

Theorems~\ref{THM:1} and~\ref{THM:2} 
 extend~\cite[Theorem 1.11]{BTT1} for the defocusing Wick ordered cubic NLS
 ($m = 2$)
to all defocusing nonlinearities (all  $m \geq 2$).
While the main structure of the argument follows that in~\cite{BTT1}, 
the main source of challenge for our problem is 
the more and more complicated
 combinatorics for higher values of $m$.
See Appendix~\ref{SEC:A}
for an example of an concrete combinatorial
argument for $m = 3$ in the case $\M=\T^2$, following the methodology in~\cite{BO96, BTT1}.
In order to overcome this combinatorial difficulty, 
we introduce the {\it white noise functional} (see Definition~\ref{DEF:W} below)
and avoid 
combinatorial arguments
of increasing complexity in~$m$,
allowing us to prove Propositions~\ref{PROP:WH1} and~\ref{PROP:nonlin}
in a concise manner.
 In order to present how we overcome the combinatorial complexity  in a clear manner, 
we decided to first discuss the proofs 
 of Propositions \ref{PROP:WH1}, \ref{PROP:Gibbs1}, 
 and \ref{PROP:nonlin}
 in the case of the flat torus $\T^2$
 (Sections~\ref{SEC:2} and~\ref{SEC:3}).
 This allows us to isolate the main idea.
We then discuss the geometric component
and prove 
the analogues of  Propositions \ref{PROP:WH1}, \ref{PROP:Gibbs1}, 
 and \ref{PROP:nonlin}
in a general geometric setting
(Section~\ref{SEC:mfd}).

\begin{remark}\label{REM:NLS}\rm

Let $m = 2$ and $\M = \T^2$.
Then,  the Wick ordered NLS~\eqref{NLS2}
can be formally written as 
\begin{align}
i \dt u + \Dl  u = (|u|^2 - 2\s_\infty) u, 
\label{NLS4}
\end{align}

\noi
where $\s_\infty$ is the (non-existent) limit of $\s_N \sim \log N $ as $N \to \infty$.

Given $u $ as in~\eqref{G5},  define
$\theta_N =  \fint_{\T^2} |\P_N u|^2 dx   - \s_N $, 
where $\fint_{\T^2} f(x)  dx = \frac 1{4\pi^2} \int_{\T^2} f(x) dx$.
Then, 
it is easy to see that 
the limit $\theta_\infty := \lim_{N \to \infty} \theta_N$ exists
in $ L^p(\mu)$ for any $p \geq 1$.
Thus, by 
 setting $v  (t) =  e^{2 i t \theta_\infty }u (t)$, 
 we can rewrite~\eqref{NLS4} as 
\begin{align}
i \dt v + \Dl  v = (|v|^2 - 2 \textstyle \fint_{\T^2} |v|^2 dx) v. 
\label{NLS5}
\end{align}

\noi
Note that $\|v\|_{L^2} = \infty$
almost surely. 
Namely,~\eqref{NLS5} is also a formal expression for the limiting dynamics.
In~\cite{BO96}, Bourgain studied 
\eqref{NLS5}
and proved local well-posedness below $L^2(\T^2)$
in a probabilistic setting.

If $v$ is a smooth solution to~\eqref{NLS5}, 
then
by  setting $w  (t) =  e^{- 2 i t \fint_{\T^2} |v|^2 dx}v(t)$, 
we see that $w$ is a solution to the standard cubic NLS:
\begin{align}
i \dt w + \Dl  w = |w|^2 w.
\label{NLS6}
\end{align}

\noi
This shows that the Wick ordered NLS~\eqref{NLS4} and~\eqref{NLS5}
are ``equivalent'' to the standard cubic NLS
in the smooth setting.
Note that  this formal reduction relies on the fact that 
the Wick ordering introduces
only  a linear term 
when $m=2$.
For $m \geq 3$, 
the Wick ordering introduces higher order terms
and thus 
there is no formal equivalence between 
the standard NLS
\eqref{NLS1}
and the Wick ordered NLS~\eqref{NLS2}.

\end{remark}

\begin{remark}\label{REM:focusing}
\rm 

So far, we focused on the defocusing NLS.
Let us now discuss the situation in the focusing case:
\begin{align*}
i \dt u + \Dl  u = - |u|^{k-2}u 
\end{align*}

\noi
with the Hamiltonian given by 
\begin{align*}
H(u) = \frac{1}{2}\int_{\T^d} |\nb u|^2dx
- \frac1{k} \int_{\T^d} |u|^k dx.
\end{align*}

In the focusing case, 
the Gibbs measure can be formally written as 
\begin{align*}
d P^{(k)}_d = 
Z^{-1} e^{-H(u)} du
= 
Z^{-1} e^{\frac 1k \int_{\T^d} |u|^k dx} d\mu. 
 \end{align*}

\noi
The main difficulty is that $\int_{\T^d} |u|^k dx$ is unbounded.
When $d = 1$, 
Lebowitz-Rose-Speer~\cite{LRS}
constructed the Gibbs measure $P^{(k)}_1$
for $2< k \leq 6$,
by adding an extra $L^2$-cutoff.
Then, Bourgain 
~\cite{BO94} 
 constructed global-in-time flow
 and proved the invariance of the Gibbs measure for $k \leq 6.$
See also McKean~\cite{McKean}.

When $ d= 2$, the situation becomes much worse. 
Indeed,  Brydges-Slade~\cite{BS}
 showed that the Gibbs measure $P^{(4)}_2$
 for the focusing cubic NLS on $\T^2$
 can not be realized as a probability measure
 even 
 with the Wick order nonlinearity
 and/or
 with a (Wick ordered) $L^2$-cutoff.
In~\cite{BO97}, Bourgain pointed out that 
an $\eps$-smoothing on the nonlinearity
makes this problem well-posed
and the invariance of the Gibbs measure
may be proven even in the focusing case.

\end{remark}

\begin{remark}\rm
In a recent paper \cite{OTNLW}, 
we also studied 
the defocusing nonlinear wave equations
 (NLW) in two spatial dimensions (with an even integer $k = 2m \geq 4$ and $\rho \geq 0$):
\begin{align}
\begin{cases}
\dt^2 u -  \Dl_g  u + \rho u  +  u^{k-1} = 0 \\
(u, \dt u) |_{t = 0} = (\phi_0, \phi_1), 
\end{cases}
\qquad (t,x) \in \R \times \M
\label{NLW1}
\end{align}

\noi
and its associated Gibbs measure:
\begin{align}
d\Pk & = Z^{-1} \exp(-  H(u,\dt u ))du\otimes d(\dt u) \notag \\
 & = Z^{-1} e^{-\frac 1{2m} \int u^{2m} dx} 
e^{-\frac 12 \int (\rho u^2 + |\nb u |^2) dx} du\otimes e^{-\frac 12 \int (\dt u)^2} d(\dt u).  
\label{NLW2}
 \end{align}

\noi
As in the case of NLS, 
the Gibbs measure in \eqref{NLW2} is not well defined in the two spatial dimensions.
Namely, 
 one needs to consider the Gibbs measure $\Pk$
associated to the Wick ordered Hamiltonian\footnote{In
the case of NLW, we only need to use the Hermite polynomials since we deal with real-valued functions.}
as in~\eqref{Gibbs2}
and study the associated dynamical problem 
given by the following  defocusing Wick ordered NLW: 
\begin{align}
\dt^2 u -  \Dl  u + \rho u \, +  :\! u^{k-1} \!:\, = 0. 
\label{WNLW2}
\end{align}

In the case of the flat torus $\M = \T^2$ with  $\rho > 0$, 
we showed that 
the  defocusing Wick ordered NLW \eqref{WNLW2}
is almost surely globally  well-posed
with respect to 
the Gibbs measure $\Pk$
and that the Gibbs measure $\Pk$ is invariant under the dynamics of \eqref{WNLW2}.
%
%
For a general two-dimensional compact Riemannian manifold without boundary or a bounded domain in $\R^2$
(with the Dirichlet or Neumann boundary condition),
we showed that an analogue of Theorem \ref{THM:2} 
(i.e.\,almost sure global existence and invariance of the Gibbs measure $\Pk$ in 
some mild sense) holds
for \eqref{WNLW2} when $\rho > 0$.
In the latter case with the Dirichlet boundary condition, 
we can also take $\rho = 0$.

In particular, our result on $\T^2$
is analogous
to that for  the defocusing cubic NLS on $\T^2$ \cite{BO96},
where the main difficulty lies in  constructing local-in-time unique solutions
almost surely with respect to the Gibbs measure.
We achieved this goal for any even $k\geq 4$ 
by exploiting one degree of smoothing
in the Duhamel formulation of the Wick ordered NLW \eqref{WNLW2}.
As for the Wick ordered NLS \eqref{NLS2} on $\T^2$, such smoothing is not available
and the construction of  unique solutions with the Gibbs measure
as initial data remains  open
for the (super)quintic case.

\end{remark}

\begin{remark}\rm 
In \cite{BO94, R}, 
Bourgain ($k = 2, 3$)
and Richards ($k = 4$)
proved invariance of the Gibbs measures
for the 
generalized KdV equation (gKdV) on the circle:
\begin{align}
\dt u + \dx^3 u = \pm \dx (u^k), 
\qquad ( t, x) \in \R \times \T.
\label{gKdV1}
\end{align}

\noi
In \cite{ORT}, 
the authors and Richards
studied the problem for $k \geq 5$.
In particular, by following the approach in \cite{BTT1}
and this paper, 
we proved almost sure global existence and invariance of the Gibbs measuresin 
some mild sense
analogous to Theorem \ref{THM:1}
for (i) all $k \geq 5$ in the defocusing case
and (ii) $k = 5$ in the focusing case.
Note that there is no need to apply a renormalization
for constructing the Gibbs measures
for this problem since the equation is posed on $\T$.
See \cite{LRS, BO94}.

\end{remark}

This paper is organized as follows.
 In Sections~\ref{SEC:2} and~\ref{SEC:3}, 
 we present the details of the proofs of Propositions \ref{PROP:WH1}, \ref{PROP:Gibbs1}, 
 and \ref{PROP:nonlin}
 in the particular case when $\M=\T^2$.
 We then indicate the changes required to treat the general case
 in Section~\ref{SEC:mfd}.
In Section \ref{Sect5}, we prove Theorems \ref{THM:1} and \ref{THM:2}.  
In Appendix \ref{SEC:A}, 
we present an alternative proof of Proposition \ref{PROP:WH1}
when $m = 3$ in the case $\M = \T^2$, 
performing concrete combinatorial computations.

\section{Construction of the Gibbs measures}
\label{SEC:2}

 In this section, we present the proofs of Propositions \ref{PROP:WH1} and \ref{PROP:Gibbs1}
and construct the Gibbs measure $\Pk$ in \eqref{Gibbs2}.	
One possible approach
is to  use the Fock-space formalism in quantum field theory
\cite{Simon, GJ, Meyer, DerezinskiGerard}.
As mentioned above, however, we present a pedestrian Fourier analytic approach to the problem
since we believe that it is more
accessible to a wide range of readers.
The argument presented in this section and the next section (on Proposition \ref{PROP:nonlin})
follows  the presentation in \cite{DPT1} with one important difference;
we work in the complex-valued setting and hence we  will make use of  
 the (generalized) Laguerre polynomials instead of the Hermite polynomials. Their orthogonal properties play an essential role.
See Lemmas \ref{LEM:W1} and \ref{LEM:Z1}.

\subsection{Hermite polynomials, 
 Laguerre polynomials, 
and Wick ordering}
\label{SUBSEC:2.1}

First, recall the Hermite polynomials $H_n(x; \s)$
defined through the generating function:
\begin{equation}
F(t, x; \s) : =  e^{tx - \frac{1}{2}\s t^2} = \sum_{k = 0}^\infty \frac{t^k}{k!} H_k(x;\s)
\label{H1}
 \end{equation}
	
\noi
for $t, x \in \R$ and $\s > 0$.
For simplicity, we set $F(t, x) : = F(t, x; 1)$
and $H_k(x) : = H_k(x; 1)$ in the following.
Note that we have 
\begin{align}
 H_k(x, \s) = \s^{\frac{k}{2}}H_k\big(\s^{-\frac{1}{2}}x \big).
\label{H1a}
 \end{align}

\noi
From~\eqref{H1},  we directly deduce   the following recursion relation 
\begin{equation}\label{H1b}
\partial_{x}H_{k}(x;\s)=kH_{k-1}(x;\s),
\end{equation}	
for all $k\geq 0$.
This allows to compute the $H_{k}$, up to the constant term. 
The constant term    is given by
\begin{equation*}
H_{2k}(0,\s)=(-1)^k (2k-1)!! \, \s^k 
\qquad \text{and} \qquad H_{2k+1}(0,\s)=0, 
\end{equation*}

\noi
for all $k\geq0$,
where $(2k-1)!! = (2k-1)(2k-3)\cdots 3\cdot 1
= \frac{(2k)!}{2^k k!}$
and $(-1)!! = 1$ by convention.	
This  can   be easily deduced from~\eqref{H1} by taking $x=0$. For readers' convenience, we write out the first few Hermite polynomials
in the following:
\begin{align*}
& H_0(x; \s) = 1, 
\qquad 
H_1(x; \s) = x, 
\qquad
H_2(x; \s) = x^2 - \s, \\
& H_3(x; \s) = x^3 - 3\s x, 
\qquad 
H_4(x; \s) = x^4 - 6\s x^2 +3\s^2.
\end{align*}

\noi
The monomial $x^k$
can be expressed in term of the Hermite polynomials:
\begin{align}
x^k = \sum_{m = 0}^{[\frac{k}{2}]}
\begin{pmatrix}
k\\2m 
\end{pmatrix}
(2m-1)!! \, \s^m H_{k-2m}(x; \s).
\label{H2}
\end{align}

Fix $d \in \N$, \footnote{Indeed, the discussion presented here also holds for $d = \infty$ 
in the context of abstract Wiener spaces.
For simplicity, however, we restrict our attention to finite values for $d$.} 
let $\H = \R^d$.
Then, consider the Hilbert space\footnote{Here, $Q_\H = \R^d$ when $d < \infty$.
When $d = \infty$, we set $Q_\H$ to be an appropriate extension of $\H$
such that $(\H, Q_\H, \mu_\infty)$ forms an abstract Wiener space
with $\H$ as the Cameron-Martin space.} $\G_\H = L^2(Q_\H, \mu_d; \C)$
endowed with  the Gaussian measure  $d\mu_d
= (2\pi)^{-\frac{d}{2}} \exp(-{|x|^2}/{2})dx$, $x = (x_1, \dots,
x_d)\in \R^d$.
We define a {\it homogeneous Wiener chaos of order
$k$} to be an element of the form 
\[ {\bf H}_{k}(x) =  \prod_{j = 1}^d H_{k_j}(x_j),\] 

\noi
where $k= k_1 + \cdots + k_d$
and $H_{k_j}$ is the Hermite polynomial of degree $k_j$ defined in~\eqref{H1}. 
Denote  by $\G_k(\H)$ the closure of homogeneous Wiener chaoses of order $k$
under $L^2(\R^d, \mu_d)$.
Then,  we have the following Wiener-Ito decomposition:\footnote{This is (equivalent to) the Fock space in 
quantum field theory.
See \cite[Chapter I]{Simon}.
In particular, 
 the Fock space $\F (\H) 
= \bigoplus_{k = 0}^\infty \H_\C^{\otimes_\text{sym}^k}
$ is shown to be equivalent to  
the Wiener-Ito decomposition \eqref{IW}. 
In the Fock space formalism, the Wick renormalization 
can be stated as the  reordering  of the creation operators on the left and annihilation operator
on the right.  
We point out that 
while much of our discussion can be recast in the Fock space formalism, 
our main aim of this paper is to give a self-contained presentation (as much as possible)
accessible to readers not familiar with the formalism
in quantum field theory.
Therefore, we stick to a simpler Fourier analytic and probabilistic approach.}
\begin{align} L^2(Q_\H, \mu_d; \C) = \bigoplus_{k = 0}^\infty \G_k(\H).
\label{IW}
\end{align}

Given a homogeneous polynomial 
$P_k(x) =P_k(x_1, \dots, x_d)$ of degree $k$, we define {\it the Wick ordered
polynomial}  $\, :\!P_k(x)\!\!: \, $ to be its projection onto
$\mathcal{H}_k$. In particular, we have $:x_j^k\!: \, = H_{k}(x_j)$ and
$:\prod_{j = 1}^d x_j^{k_j}\!:\, = \prod_{j = 1}^d H_{k_j}(x_j)$ 
with $k= k_1 + \cdots + k_d$.

Now,  let $g$ be a standard complex-valued Gaussian random
variable. Then, $g$ can be written as $g = \frac{h_1}{\sqrt{2}}+ i\frac{h_2}{\sqrt{2}}$, 
where $h_1$ and $h_2$
are independent standard real-valued Gaussian random variables. 
We
investigate the Wick ordering on $|g|^{2m}$ for $m\in \mathbb{N}$,
that is, the projection of $|g|^{2m}$ onto $\mathcal{H}_{2m}$.
When $m = 1$, $|g|^2 = \frac 12 (h_1^2 + h_2^2)$ is Wick-ordered into
\begin{align}
 :\!|g|^2\!: \, = \tfrac 12 (h_1^2 - 1) + \tfrac 12 (h_2^2-1) = |g|^2 - 1.
\label{H2a}
 \end{align}

\noi
When $m = 2$, $|g|^4 = \frac 14 (h_1^2+ h_2^2)^2 = \frac 14 (h_1^4 + 2  h_1^2 h_2^2 +  h_2^4)$ is
Wick-ordered into
\begin{align*}
:\!|g|^4\!:\,  & = \tfrac 14 (h_1^4 -6 h_1^2 + 3) + \tfrac 12 (h_1^2 - 1)(h_2^2 - 1)
+ \tfrac 14(h_2^4 -6 h_2^2 + 3)\\
& = \tfrac 14(h_1^4 + 2 h_1^2 h_2^2 + h_2^4)
- 2 (h_1^2 + h_2^2) + 2 \\
& = |g|^4 - 4  |g|^2 + 2.
\end{align*}

\noi
When $m = 3$, 
a direct computation shows that
 $$|g|^6 = \frac 18 (h_1^2+ h_2^2)^3 = \frac 18(h_1^6 + 3 h_1^4 h_2^2 +3h_1^2 h_2^4 +h_2^6)$$
  is Wick-ordered into
\begin{align*}
:\!|g|^6\!: \, & =  \tfrac 18 H_6(h_1) 
+ \tfrac 38 H_4(h_1) H_2(h_2)
+ \tfrac 38 H_2(h_1) H_4(h_2)
+ \tfrac 18 H_6(h_2) \\
& = |g|^6 - 9|g|^4 + 18 |g|^2 - 6.
\end{align*}

\noi
In general, we have 
\begin{align}
:\! |g |^{2m} \!: \, 
 & = \frac 1{2^m}\sum_{\l = 0}^m 
\begin{pmatrix}
m\\ \l
\end{pmatrix}
H_{2\l}(h_1 ) H_{2m - 2\l}(h_2)\notag \\
&  = \sum_{\l = 0}^m 
\begin{pmatrix}
m\\ \l
\end{pmatrix}
H_{2\l}(\Re g; \tfrac 12 ) H_{2m - 2\l}(\Im g; \tfrac 12 ),
 \label{H3}
\end{align}

\noi
where we used~\eqref{H1a} in the second equality.
It follows from the rotational invariance of the 
complex-valued Gaussian random variable
that 
$:\! |g |^{2m} \!: \, = P_m(|g|^2)$
for some polynomial $P_m$ of degree $m$
with the leading coefficient 1.
This fact is, however, not obvious from~\eqref{H3}.

The following lemma shows that the 
Wick ordered monomials 
$:\! |g |^{2m} \!: $
can be expressed in terms of the Laguerre polynomials (recall the definition~\eqref{L1}).

\begin{lemma}\label{LEM:Wick1}
Let $m \in \N$.
For  a complex valued mean-zero
Gaussian random variable $g$ with $\textup{Var}(g) = \s>0$, 
we have 
\begin{align}
:\! |g |^{2m} \!: \, 
&  = \sum_{\l = 0}^m 
\begin{pmatrix}
m\\ \l
\end{pmatrix}
H_{2\l}(\Re g; \tfrac \s2 ) H_{2m - 2\l}(\Im g; \tfrac \s2 )\notag \\
& 
= (-1)^m m!\cdot L_m(|g|^2; \s).
 \label{H4}
\end{align}

As a consequence, 
 the Wick ordered 
monomial $:\! |u_N|^{2m}  \!:$
defined in~\eqref{L0}
satisfies~\eqref{L3}
for any $N \in \N$.

\end{lemma}

\begin{proof}
The first equality follows from~\eqref{H3}
and scaling with~\eqref{H1a}.
Moreover, by scaling with~\eqref{L2} and~\eqref{H1a}, 
we can assume that 
$g$ is a standard complex-valued Gaussian random variable
with $g_1 = \Re g$ and $g_2 = \Im g$.
Define $\mathfrak{H}_m(|g|^2) $ and $\mathfrak{L}_m(|g|^2) $ by 
\begin{align}
\mathfrak{H}_m(|g|^2) &  
= \sum_{\l = 0}^m 
\begin{pmatrix}
m\\ \l
\end{pmatrix}
H_{2\l}(g_1; \tfrac 12) H_{2m - 2\l}(g_2 ; \tfrac 12), \notag \\
\mathfrak{L}_m(|g|^2) & = (-1)^m m!\cdot L_m(|g|^2).
 \label{H5}
\end{align}

\noi
Then,~\eqref{H4} follows once we prove
the following three properties:
\begin{align}
& \hphantom{X}
\mathfrak{H}_1(|g|^2)
 = \mathfrak{L}_1(|g|^2)
= |g|^2 - 1, \label{H6}\\
&  \begin{cases}
\vphantom{\Big|}
\frac{\dd^2}{\dd g \dd \cj g}
\mathfrak{H}_m(|g|^2)
= m^2 
 \mathfrak{H}_{m-1} (|g|^2),\\
\vphantom{\Big|}
\frac{\dd^2}{\dd g \dd \cj g}
\mathfrak{L}_m(|g|^2)
= 
m^2 \mathfrak{L}_{m-1}(|g|^2), 
\end{cases}
 \label{H7}\\
& \hphantom{X}
 \E[\mathfrak{H}_m(|g|^2)]
 = \E[\mathfrak{L}_m(|g|^2)] = 0\label{H8},
\end{align}

\noi
for all $m \geq 2$.
Noting that both 
$\mathfrak{H}_m(|g|^2)$
and $\mathfrak{L}_m(|g|^2)$
are polynomials in $|g|^2$, 
the three properties~\eqref{H6}, 
\eqref{H7}, and~\eqref{H8}
imply that 
$\mathfrak{H}_m(|g|^2) = \mathfrak{L}_m(|g|^2)$
for all $m \in \N$.

The first property~\eqref{H6}
follows from~\eqref{H2a} and 
\eqref{L1a}.
Next, we prove~\eqref{H7} for $\mathfrak{H}_m(|g|^2)$.
From $\dd_{g} = \frac 12( \dd_{g_1} - i \dd_{g_2})$
and
$\dd_{\cj g} = \frac 12( \dd_{g_1} + i \dd_{g_2})$, 
we have 
\begin{align*}
\frac{\dd^2}{\dd g \dd \cj g}
= \frac{1}{4} \Dl_{g_1, g_2}, 
\end{align*}

\noi
where $ \Dl_{g_1, g_2}$	
denotes the usual Laplacian on $\R^2$ in the variables
$(g_1, g_2)$.
Then, recalling that $\dx H_k(x; \s) = k H_{k-1}(x; \s)$, 
we have
\begin{align}
\frac{\dd^2}{\dd g \dd \cj g}
& \mathfrak{H}_m(|g|^2)
 = \frac{1}{4} \Dl_{g_1, g_2}
\mathfrak{H}_m(|g|^2)\notag \\
& = 
\frac 14 \sum_{\l = 1}^m 
\begin{pmatrix}
m\\ \l
\end{pmatrix}
2\l (2\l-1)
H_{2\l-2}(g_1; \tfrac 12) H_{2m - 2\l}(g_2 ; \tfrac 12)\notag \\
& \hphantom{X}
+ \frac 14 
\sum_{\l = 0}^{m-1} 
\begin{pmatrix}
m\\ \l
\end{pmatrix}
(2m -2\l) (2m - 2\l - 1)
H_{2\l}(g_1; \tfrac 12) H_{2m - 2\l - 2}(g_2 ; \tfrac 12)\notag \\
& 
= m^2 \sum_{\l = 0}^{m-1} 
\begin{pmatrix}
m - 1\\ \l
\end{pmatrix}
H_{2\l}(g_1; \tfrac 12) H_{2(m-1) - 2\l}(g_2 ; \tfrac 12). \notag
\end{align}
	
\noi
As for the second identity in~\eqref{H7}, thanks to the formula~\eqref{L1b},  we get
\begin{align*}
\frac{\dd^2}{\dd g \dd \cj g}
 \mathfrak{L}_m(|g|^2)
&  = \frac{(-1)^m m!}{4} 
\sum_{\l = 0}^m
\begin{pmatrix}
m\\ \l
\end{pmatrix}
\frac{(-1)^\l}{\l!} \Dl_{g_1, g_2} (g_1^2 + g_2^2)^\l \\
& 
 = (-1)^{m-1} m!
\sum_{\l = 1}^m
\begin{pmatrix}
m\\ \l
\end{pmatrix}
\frac{(-1)^{\l-1}}{\l!} \l^2 |g|^{2(\l-1)}
= m^2 \mathfrak{L}_{m-1}(|g|^2).
\end{align*}

\noi
This proves~\eqref{H7}.
The property~\eqref{H8} follows
from (i) independence of $g_1$ and $g_2$
together with 
the orthogonality of 
$H_k(x)$ and the constant function 1
under $e^{-x^2} dx$
and (ii) the orthogonality of 
$L_m(x)$ and the constant function 1
under $\ind_{\R_+}e^{-x} dx$

Let $u$ be as in~\eqref{G5}.
Fix  $ x \in \T^2$.
Letting $\wt g_n = g_n e^{in\cdot x}$, 
we see   that $\{\wt g_n\}_{n \in \N}$  is 
 a sequence of independent standard complex-valued Gaussian
random variables.
Then, given $N \in \N$, 
$\Re u_N(x)$ and $\Im u_N(x)$
are mean-zero real-valued Gaussian
random variables with variance $\frac{\s_N}{2}$, 
while 
$u_N(x)$ is a mean-zero complex-valued Gaussian
random variable with variance $\s_N$, 
Then, 
it follows from~\eqref{L0} with~\eqref{Wick1}
and~\eqref{H4} that
\begin{align*}
:\! |u_N(x) |^{2m} \!: \, 
& = \sum_{\l = 0}^m 
\begin{pmatrix}
m\\ \l
\end{pmatrix}
H_{2\l}(\Re u(x); \tfrac {\s_N}2) H_{2m - 2\l}(\Im u(x) ; \tfrac  {\s_N}2) \notag \\
& = (-1)^m m!\cdot L_m(|u_N(x)|^2; \s_N),
\end{align*}

\noi
verifying~\eqref{L3}.
This proves the second claim in Lemma~\ref{LEM:Wick1}.
\end{proof}

\subsection{White noise functional}
Next, we define the white noise functional.
Let  $w(x;\o)$ be the  mean-zero complex-valued Gaussian white noise on $\T^2$
defined by
\[ w(x;\o) = \sum_{n\in \Z^2} g_n(\o) e^{in\cdot x}.\]

\begin{definition}\label{DEF:W}\rm
The white noise functional $W_{(\cdot)}: L^2(\T^2) \to L^2(\O)$
is defined by 
\begin{equation}
 W_f (\o) = \jb{f, w(\o)}_{L^2_x} = \sum_{n \in \Z^2} \ft f(n) \cj{g_n}(\o).
\label{W0}
 \end{equation}

\noi
for a function $f \in L^2(\T^2)$.

\end{definition}

Note that this is basically the periodic
and higher dimensional version of the classical Wiener integral
$\int_a^b f dB$.
It can also be viewed as the Gaussian process indexed by $f \in L^2(\T^2)$.
See \cite[Model 1 on p.\,19 and Model 3 on p.\,21]{Simon}.
For each $f \in L^2(\T^2)$,  
$W_f$ is a complex-valued Gaussian random variable
with mean 0 and variance $\|f\|_{L^2}^2$.
Moreover, we have 
\[ E\big[ W_f \cj{W_h} ] = \jb{f, h}_{L^2_x}\]

\noi
for $f, h \in L^2(\T^2)$.
In particular, the white noise functional
$W_{(\cdot)}: L^2(\T^2) \to L^2(\O)$
is an isometry.

\begin{lemma}\label{LEM:W0}
Given $f \in L^2(\T^2)$, we have
\begin{align}
\int_{\O} e^{\Re W_f(\o)} dP(\o)
= e^{\frac{1}{4}\|f\|_{L^2}^2}.
\label{W0a}
\end{align}

\end{lemma}

\begin{proof}
Noting that $\Re g_n$ and $\Im g_n$
are mean-zero real-valued Gaussian random variables with variance $\frac 12$,
it follows from~\eqref{W0} that 
\begin{align*}
\int_{\O} e^{\Re W_f(\o)} dP(\o)
& = \prod_{n \in \Z} \frac 1{\pi} \int_\R e^{\Re \ft f(n) \Re g_n - (\Re g_n)^2} d \Re g_n\\
& \hphantom{XXXXX}
\times
 \int_\R e^{  \Im \ft f(n) \Im g_n - (\Im g_n)^2} d \Im g_n
\\
& = e^{\frac{1}{4}\|f\|_{L^2}^2}.
\qedhere
\end{align*}

\end{proof}

The following lemma on the white noise functional
and the Laguerre polynomials plays an important role in our analysis.
In the following, we present an elementary proof, using the
generating function $G$ in \eqref{L1}.
See also Folland \cite{Folland}.

\begin{lemma}\label{LEM:W1}
Let $f, h \in L^2(\T^{2})$ such that $\|f\|_{L^2} = \|h\|_{L^2} = 1$.
Then, for $k, m \in \Z_{\geq 0}$, we have 
\begin{align}
\E\big[ L_k(|W_f|^2)L_m(|W_h|^2)\big]
=  \dl_{km}  |\jb{f, h}|^{2k}.
\label{W1}
\end{align}

\noi
Here, $\dl_{km}$ denotes the Kronecker delta function.
\end{lemma}

First, recall the following identity:	
\begin{align}
e^\frac{u^2}{2} = \frac{1}{\sqrt{2\pi}} \int_\R e^{xu - \frac{x^2}{2}} dx. 
\label{W1a}
\end{align}

\noi
Indeed, we used a rescaled version of~\eqref{W1a}
in the proof of Lemma~\ref{LEM:W0}.

\begin{proof}[Proof of Lemma~\ref{LEM:W1}]
Let $G$ be as in~\eqref{L1}.
Then, for any $ -1< t, s <0$, 
from~\eqref{W1a} and Lemma~\ref{LEM:W0}, we have 
\begin{align}
\int_{\O} & G(t, |W_f(\o)|^2) G(s, |W_h(\o)|^2) dP(\o)
 =  \frac{1}{1-t}\frac1 {1-s} 
\int_\O
e^{-\frac{t}{1-t}|W_f|^2 - \frac{s}{1-s}|W_h|^2}
dP(\o)\notag \\
& =  \frac{1}{1-t}\frac1 {1-s} 
\frac1 {4\pi^2}
\int_{\R^4} e^{-\frac{x_1^2 + x_2^2 + y_1^2 + y_2^2}{2}}\notag \\
& \hphantom{XXXXXXXX}
\times 
\int_\O
\exp\Big(
\Re W_{
\sqrt{\frac{-2t}{1-t}}(x_1 -i  x_2 ) f
+ \sqrt{\frac{-2s}{1-s}}(y_1 -i y_2) h}\Big)
dP dx_1dx_2dy_1dy_2  \notag \\
& =  \frac{1}{1-t}\frac1 {1-s} 
\frac1 {4\pi^2}
\int_{\R^4} e^{-\frac{x_1^2 + x_2^2 }{2(1-t)}-\frac{y_1^2 + y_2^2 }{2(1-s)}}
\notag \\
& \hphantom{XXXXXXXX}
\times 
e^{\frac 12 \Re \big(
\sqrt{\frac{-2t}{1-t}} \sqrt{\frac{-2s}{1-s}}(x_1 -i  x_2 ) 
 (y_1 + i y_2)\jb{f,  h}\big)}
 dx_1dx_2dy_1dy_2  \notag \\
 \intertext{By a change of variables
 and applying~\eqref{W1a}, we have} 
& =  
\frac1 {4\pi^2}
\int_{\R^2} e^{-\frac{ y_1^2 + y_2^2}{2}}
\int_\R e^{ \sqrt{ts}
 (y_1 \Re \jb{f, h} - y_2 \Im \jb{f,  h})x_1 - \frac{x_1^2}{2}} dx_1 \notag \\
& \hphantom{XXXXXXXX}
\times \int_\R
 e^{ \sqrt{ts}
 (y_2 \Re \jb{f, h} + y_1 \Im \jb{f,  h})x_2- \frac{x_2^2}{2}} dx_2 dy_1dy_2  \notag \\
& =  
\frac1 {2\pi}
\int_{\R^2} e^{-\frac{ y_1^2 + y_2^2}{2}}
e^{ \frac 1 2 ts |\jb{f, h}|^2 (y_1^2 + y_2^2) }dy_1dy_2 \notag \\
& = \frac{1}{1-ts |\jb{f, h}|^2}
= \sum_{k = 0}^\infty t^ks^k |\jb{f, h}|^{2k}.
\label{W1b}
\end{align}

\noi
In the second to the last equality, we used the fact that 
$\frac 12 ts |\jb{f, h}|^2 < \frac 12$.
Hence, it follows from~\eqref{L1} and~\eqref{W1b} that 
\begin{align*}
\sum_{k = 0}^\infty t^ks^k |\jb{f, h}|^{2k}
 = \sum_{k, m  = 0}^\infty 
t^ks^m
\int_\O L_k(|W_f(\o)|^2)L_m(|W_h(\o)|^2)dP(\o).
\end{align*}

\noi
By comparing the coefficients of $t^ks^m$, we obtain~\eqref{W1}.
\end{proof}

Now, we are ready to make sense of the 
nonlinear part of the Wick ordered Hamiltonian $H_{\text{\tiny Wick}}$ in~\eqref{Hamil2}.
We first present the proof of Proposition~\ref{PROP:WH1}
for $p = 2$. Recall that
\begin{align*}
G_N(u) = \frac 1{2m} \int_{\T^2}:\! |\P_N u|^{2m}\!: dx.
\end{align*}

\noi
Then, we have the following convergence property of $G_N(u)$ in $L^2(\mu)$.

\begin{lemma}\label{LEM:WH2}

Let $m \geq2$ be an   integer.
Then, $\{G_N(u)\}_{N\in \N}$
is a Cauchy sequence
in $L^2(H^s(\T^2), \mu)$.
More precisely, 
there exists $C_{m} > 0$ such that 
\begin{align}
\| G_M(u) - G_N(u) \|_{L^2(\mu)}
\leq \frac{C_m}{N^\frac{1}{2}}
\label{WH2}
\end{align}

\noi
for 
any $ M \geq N \geq 1$.

\end{lemma}

Given $N \in \N$, 
let $\s_N$ be as in~\eqref{Wick4}.
For {\it fixed} $x \in \T^{2}$ and 
 $N \in \N$, 
 we  define
\begin{align}
\eta_N(x) (\cdot) & := \frac{1}{\s_N^\frac{1}{2}}
\sum_{ |n| \leq N} \frac{\cj{e_n(x)}}{\sqrt{1+ |n|^2}}e_n(\cdot), 
\label{W3}\\
\g_N (\cdot) & := 
\sum_{ |n| \leq N} \frac{1}{1+ |n|^2}e_n(\cdot),
\label{W3a}
\end{align}
	
\noi
where $e_n(y) = e^{in\cdot y}$.
Note that 
\begin{align}
 \| \eta_N(x)\|_{L^2(\T^{2})} = 1
\label{W3b}
\end{align}	

\noi
for all (fixed) $x \in \T^{2}$ and all $N \in \N$.
Moreover, we have 
\begin{align}
\jb{\eta_M(x), \eta_N(y)}_{L^2(\T^{2})}
= \frac{1}{\s_M^\frac{1}{2}\s_N^\frac{1}{2}} \g_N(y-x)
= \frac{1}{\s_M^\frac{1}{2}\s_N^\frac{1}{2}} \g_N(x-y), 
\label{W4}
\end{align}

\noi
for fixed $x, y\in \T^{2}$
and $N, M \in \N $ with $M\geq N$.

\begin{proof}[Proof of Lemma~\ref{LEM:WH2}]

Let $m \geq 2$ be an  integer.
Given $N \in \N$ and $x \in \T^{2}$, 
it follows from~\eqref{G5},~\eqref{W0}, and~\eqref{W3}
that 
\begin{align}
u_N (x) = \s_N^\frac{1}{2}\frac{u_N(x)}{\s_N^\frac{1}{2}}
= \s_N^{\frac 12} \cj{W_{ {\eta_N(x)}}}.
\label{W6}
\end{align}

\noi	
Then, from~\eqref{L3} and~\eqref{W6}, we have 
\begin{align}
:\! |u_N|^{2m} \!: \, =
 (-1)^m m!   \s^m_NL_m\bigg(\frac{| u_N|^2}{\s_N}\bigg)
 \, = (-1)^m m! \s_N^m L_m\big(\big|W_{ {\eta_N(x)}}\big|^2\big).
 \label{W7}
\end{align}

From 
~\eqref{W7}, Lemma~\ref{LEM:W1}, and~\eqref{W4}, 
we have
\begin{align}
(2m)^{2} \|  G_M(u)   - & G_N(u)  \|_{L^2(\mu)}^2
  = (m!)^2 
 \int_{\T^2_x\times \T^2_y}
\int_{\O} 
\Big[ 
 \s_M^{2m} L_m\big(\big|W_{ {\eta_M(x)}}\big|^2\big)
L_m\big(\big|W_{  {\eta_M(y)}}\big|^2\big) \notag \\
& \hphantom{XXXXXXXXXX}
 -  \s_M^{m}\s_N^{m} L_m\big(\big|W_{ {\eta_M(x)}}\big|^2\big)
L_m\big(\big|W_{ {\eta_N(y)}}\big|^2\big)\notag \\
& \hphantom{XXXXXXXXXX}
 -   \s_M^{m}\s_N^{m} L_m\big(\big|W_{ {\eta_N(x)}}\big|^2\big)
L_m\big(\big|W_{ {\eta_M(y)}}\big|^2\big) \notag \\
& \hphantom{XXXXXXXXXX}
 +  \s_N^{2m} L_m\big(\big|W_{ {\eta_N(x)}}\big|^2\big)
L_m\big(\big|W_{{\eta_N(y)}}\big|^2\big)
\Big] dP dx dy
 \notag \\
 & = (m!)^2 
 \int_{\T^2_x\times \T^2_y}   
 \big[ 
 (\g_M(x-y))^{2m} - (\g_N(x-y))^{2m}\big] dx dy
 \notag \\
 & = (m!)^2 
 \int_{\T^2}   
 \big[ 
 (\g_M(x))^{2m} - (\g_N(x))^{2m}\big] dx 
 \notag \\
& \leq C_m 
 \int_{\T^2}
\big|\g_M(x)-\g_N(x)\big|
\cdot  \big[|\g_M(x)|^{2m-1} + |\g_N(x)|^{2m-1}\big] dx .
\label{W8a}
\end{align}

\noi
In the second equality, we used the fact that $\g_N$ is
a real-valued function.

From~\eqref{W3a}, we have
\begin{align}
\big\|\g_M-\g_N\big\|_{L^2}
= \bigg(\sum_{N < |n|\leq M} \frac{1}{(1+|n|^2)^2}\bigg)^\frac{1}{2}
\les \frac{1}{N}.
\label{W8b}
\end{align}
	
\noi
By Hausdorff-Young's inequality, we have
\begin{align}
\big\| |\g_N|^{2m-1}\big\|_{L^2}
= \|\g_N\|_{L^{4m-2}}^{2m-1}
\leq \Bigg(\sum_{|n|\leq N} \frac{1}{(1+|n|^2)^{\frac{4m-2}{4m-3}}}\Bigg)^{\frac{4m-3}{2}}
\leq C_m < \infty
\label{W8c}
\end{align}

\noi
uniformly in $N \in \N$.
Then, 
\eqref{WH2} follows from~\eqref{W8a},~\eqref{W8b}, and~\eqref{W8c}.
\end{proof}

\subsection{Wiener chaos estimates}\label{sect23}
In this subsection, we complete the proof of Proposition
~\ref{PROP:WH1}.
Namely, we upgrade~\eqref{WH2} in Lemma~\ref{LEM:WH2}
 to any finite $p \geq 2$.
Our main tool is  the following Wiener chaos estimate
(see \cite[Theorem I.22]{Simon}).

\begin{lemma}\label{LEM:hyp3}
 Let $\{ g_n\}_{n \in \N }$ be 
 a sequence of  independent standard real-valued Gaussian random variables.
Given  $k \in \mathbb{N}$, 
let $\{P_j\}_{j \in \N}$ be a sequence of polynomials in 
$\bar g = \{ g_n\}_{n \in \N }$ of  degree at most $k$. 
Then, for $p \geq 2$, we have
\begin{equation}
 \bigg\|\sum_{j \in \N} P_j(\bar g) \bigg\|_{L^p(\O)} \leq (p-1)^\frac{k}{2} \bigg\|\sum_{j \in \N} P_j(\bar g) \bigg\|_{L^2(\O)}.
  \label{hyp4}
 \end{equation}

\end{lemma}

Observe that the estimate \eqref{hyp4} is independent of $d \in\N$. 
By noting that $P_j(\bar g)\in  \bigoplus_{\l = 0}^k \G_\l(\H)$, 
this lemma follows as a direct corollary to the
  hypercontractivity of the Ornstein-Uhlenbeck
semigroup due to Nelson \cite{Nelson2}.

We are now ready to present the proof of Proposition~\ref{PROP:WH1}.

\begin{proof}[Proof of Proposition~\ref{PROP:WH1}]

Let $m \geq 2$ be an  integer.
For $1\leq p \leq 2$, 
Proposition~\ref{PROP:WH1} follows from 
Lemma~\ref{LEM:WH2}.
In the following, we consider the case $p > 2$.
From~\eqref{L1b},~\eqref{L3}, and~\eqref{L4}, we have 
\begin{align*}
G_M(u) - G_N(u) 
=\frac{ (-1)^m m!}{2m}
\sum_{\l = 1}^m
\begin{pmatrix}
m\\ \l
\end{pmatrix}
\frac{(-1)^\l}{\l!} 
\Si_{\l}. 
\end{align*}

\noi
Here,  $\Si_\l$ is given by 
\[ \Si_{\l} = \frac{ \s^{m}_{M}}{\s^{\ell}_{M}}\sum_{\substack{\G_{2\l}(0)\\|n_j| \leq M}}\prod_{j = 1}^{2\l} \frac{g^*_{n_j}}{\sqrt{1+|n_j|^2}}
-  \frac{ \s^{m}_{N}}{\s^{\ell}_{N}}\sum_{\substack{\G_{2\l}(0)\\|n_j| \leq N}}
\prod_{j = 1}^{2\l} \frac{g^*_{n_j}}{\sqrt{1+|n_j|^2}},\]

\noi
where $\G_k$ and   $g^*_{n_j}$ are defined by 
\begin{align}
& \G_k(n) = \{ (n_1, \dots, n_{k}) \in \Z^{k}: n_1 - n_2 + \cdots +(-1)^k n_{k} = n\}, 
\label{N1a}\\
& g^*_{n_j} =
\begin{cases}
 g_{n_j} & \text{if $j$ is odd,} \\
\cj {g_{n_j}} & \text{if $j$ is even.}
\end{cases}
\label{N1b}
\end{align}

\noi
Noting that $\Si_{\l}$ is a sum of polynomials
of degree $2\l$ in $\{g_n\}_{n \in \Z^2}$, 
Proposition~\ref{PROP:WH1} follows
from Lemmas~\ref{LEM:WH2} and~\ref{LEM:hyp3}.
\end{proof}

\subsection{Nelson's estimate}

In this subsection, we prove Proposition~\ref{PROP:Gibbs1}.
Our main tool is the so-called Nelson's estimate,
i.e.~in establishing 
an tail estimate of size $\ld>0$, we divide the argument into
low and high frequencies, depending on the size of $\ld$.
See~\eqref{N1} and~\eqref{N3}.
What plays a crucial role here is
the defocusing property of the Hamiltonian
and the logarithmic upper bound on $-G_N(u)$, which we discuss below.

For each $m \in \N$, there exists finite  $a_m >0$ such that 
$(-1)^m L_m(x^2) \geq - a_m $ 
for all $x \in \R$.
Then, it follows from ~\eqref{L2}, 
\eqref{L3}, \eqref{Wick4}, and~\eqref{L4} that there exists some finite $b_m > 0$ 
such that 
\begin{align}
-  G_N(u) =  -  \frac{1}{2m} \int_{\T^2}:\! |\P_N u|^{2m}\!: dx \leq   b_m (\log N)^{m}
\label{N2}
\end{align}

\noi
 for all $N\geq 1$.
Namely, while $G_N(u)$ is not sign definite, 
$-G_N(u)$ is bounded from above by a power of $\log N$.
This is where 
the defocusing property of the equation
\eqref{NLS3} plays an essential role.

\begin{proof}[Proof of Proposition~\ref{PROP:Gibbs1}]
Let $m \geq 2$ be an integer.
It follows from 	
 Proposition~\ref{PROP:WH1}
that the following tail estimate holds:
there exist $c_{m,p}, C_m > 0$ such that 
\begin{align}
 \mu\big(p | G_M(u) - G_N(u)| > \ld \big)
 \leq C_m e^{ - c_{m, p} N^\frac{1}{2m}\ld^\frac{1}{m}}
\label{N1}
\end{align}

\noi
for all $ M \geq N \geq 1$, $p \geq 1$, 
and all $\ld > 0$.
See,  for example,~\cite[Lemma 4.5]{Tzv}.

 We first show that $R_N(u)=e^{- G_{N}(u)}$ is in $ L^p(\mu)$ with a uniform bound in $N$. We have
\begin{align*}
\|R_{N}(u)\|^{p}_{L^{p}(\mu)}
&= \int_{H^{s}} e^{-pG_{N}(u)}d\mu(u)\\
&=\int_{0}^{\infty}\mu(e^{-pG_{N}(u)}> \al)d \al\\
&\leq 1+\int_{1}^{\infty}\mu(-pG_{N}(u)>\log \al)d \al.
\end{align*}

\noi
Hence,  it suffices to show that 
there exist $C,\dl>0$ such that 
\begin{equation}\label{N11}
\mu(-pG_{N}(u)>\log \al)\leq C\al^{-(1+\dl)}
\end{equation}

\noi
for all $\al > 1$ and $N \in \N$.
Given  $\ld =\log \al> 0$, 
 choose $N_0\in \R$ such that 
$\ld  = 2 p b_m (\log N_0)^{m} $.
Then,  it follows from~\eqref{N2} that 
\begin{align}
\mu \big( -p G_N(u) > \ld \big)
= 0
\label{N3}
\end{align}

\noi
for all $N < N_0$.
For $N \geq N_0$, it follows  from~\eqref{N2} and~\eqref{N1}
that 
there exist $\dl_{m, p}>0$ and $  C_{m, p}> 0$ such that 
\begin{align}
\mu \big( -p G_N(u) > \ld \big)
& \leq 
\mu \big( -p G_N(u) + p G_{N_0}(u) > \ld   -p  b_m (\log N_0)^{m} \big)\notag \\
& \leq  \mu \big( -p G_N(u) + p G_{N_0}(u) > \tfrac 12\ld   \big)\notag \\
& \leq C_m e^{ - c'_{m, p} N_0^\frac{1}{2m}\ld^\frac{1}{m}}
 =    C_m e^{ - c'_{m, p} \ld^\frac{1}{m} e^{\wt c_m \ld^\frac{1}{m}}}\notag\\
 & \ll C_{m, p} e^{-{(1+ \dl_{m, p})\lambda}}
\label{N4}
\end{align}
	
\noi
for all $N \geq N_0$.
 This shows that \eqref{N11} is satisfied
 in this case as well.
Hence, we have  $R_N(u) \in L^p(\mu)$   with a uniform bound in $N$,  depending on $p \geq 1$.

By \eqref{N1}, $G_{N}(u)$ converges to $G(u)$ in measure with respect to $\mu$.
Then, as a composition of $G_N(u)$ with  a continuous function, 
$R_{N}(u)=e^{-G_{N}(u)}$ converges to $R(u):=e^{-G(u)}$ in measure with respect to $\mu$. 
In other words,  given $\eps > 0$,  defining $A_{N, \eps}$ by 
\begin{equation*}
A_{N,\eps}=\big\{ \, |R_{N}(u)-R(u)| \leq\eps  \,\big\}, 
\end{equation*}

\noi
we have  $\mu({A^{c}_{N,\eps}})\to 0$, as $N\to\infty$. 
Hence,  by Cauchy-Schwarz inequality and the fact that $ \|R\|_{L^{2p}},   \|R_{N}\|_{L^{2p}}\leq C_{p}$
uniformly in $N \in \N$, we obtain
\begin{align*}
\|R-R_{N}\|_{L^{p}(\mu)}
&\leq \|(R-R_{N}){\bf 1}_{A_{N,\eps}}\|_{L^{p}(\mu)}+\|(R-R_{N}){\bf 1}_{{A^{c}_{N,\eps}}}\,\|_{L^{p}(\mu)}\\
&\leq  \eps \big(\,\mu(A_{N,\eps}\,)\big)^{\frac1p}+\|R-R_{N} \|_{L^{2p}(\mu)}
\big(\,\mu({A^{c}_{N,\eps}})\,\big)^{\frac1{2p}}\leq C\eps,
\end{align*}

\noi
for all sufficiently large $N$. This completes the proof  of Proposition~\ref{PROP:Gibbs1}.
\end{proof}

\section{On the Wick ordered nonlinearity}
\label{SEC:3}

In this section, we present the proof of Proposition~\ref{PROP:nonlin}.
The main idea is similar to that in Section~\ref{SEC:2}
but, this time,  we will make use of the generalized Laguerre functions $L_m^{(\al)}(x)$.
The generalized Laguerre polynomials $L^{(\al)}_m(x)$
are 
defined through the following generating function:
\begin{equation}
G_\al(t, x) : = \frac{1}{(1-t)^{\al+1}} e^{- \frac{tx}{1-t}} = \sum_{m = 0}^\infty t^mL_m^{(\al)}(x),
\label{T1}
 \end{equation}
	
\noi
for $|t| < 1$ and $x\in \R$. From~\eqref{T1},  
we obtain the following differentiation rule; for $\l \in \N$, 
\begin{align}
\frac{d^\l}{dx^\l} L^{(\al)}_m(x) = (-1)^\l L^{(\al+\l)}_{m-\l}(x).
\label{T2}
\end{align}

Given $N \in \N$, 
let $u_N = \P_N$, where
 $u$ is as in~\eqref{G5}.
Let $ m \geq2$ be an   integer.
Then, from~\eqref{nonlin2},~\eqref{L3}, 
\eqref{L2}, and~\eqref{T2}, we have 
\begin{align}
 F_N(u) 
 & = \P_N\big(\!   :\!|\P_N    u|^{2(m-1)} \P_Nu\!: \!\big)
  = (-1)^m m! \s_N^m \cdot \tfrac{1}{m}\P_N \dd_{\cj u_{N}} \Big\{L_m\Big(\tfrac{|u_N|^2}{\s_N}\Big)\Big\}\notag \\
& = (-1)^{m+1} (m-1)! \s_N^{m-1} \cdot \P_N \Big\{  L_{m-1}^{(1)}\Big(\tfrac{|u_N|^2}{\s_N}\Big) u_N \Big\}.
\label{T3}
\end{align}

\noi
\begin{remark}\rm 
Here, $\dd_{\cj u}$ denotes the usual differentiation in $\cj u$
 viewing $u$ and $\cj u$ as independent variables.
 This is not to be confused with 
$\frac{\dd H}{\dd \cj u}$ in~\eqref{Hamil0}.
Note that $\frac{\dd H}{\dd \cj u}$ in~\eqref{Hamil0}
comes from the symplectic structure of NLS
and the G\^ateaux derivative of $H$.
More precisely, we can view the dynamics
of NLS~\eqref{NLS1}
as a Hamiltonian dynamics with the symplectic space $L^2(\T^2)$
and the symplectic form $\dis \o(f, g) = \Im \int f(x) \cj{g(x)} dx$.
Then, we define $\frac{\dd H}{\dd \cj u}$
by 
\[ dH|_u(\phi) = \o\Big(\phi, -i \tfrac{\dd H}{\dd \cj u}\Big),\]

\noi
where 
$dH|_u(\phi)$
is  the 
 the G\^ateaux derivative
 given by
$dH|_u(\phi) = \frac{d}{d\eps}H(u + \eps \phi)\big|_{\eps = 0}$.
\end{remark}

The following lemma 
is an analogue of Lemma~\ref{LEM:W1}
for  the generalized Laguerre polynomials $L^{(1)}_m(x)$
and plays an important role in the proof of Proposition~\ref{PROP:nonlin}.

\begin{lemma}\label{LEM:Z1}
Let $f, h \in L^2(\T^{2})$ such that $\|f\|_{L^2} = \|h\|_{L^2} = 1$.
Then, for $k, m \in \Z_{\geq 0}$, we have 
\begin{align}
\E\Big[ \Li_k(|W_f|^2)W_f \cj{\Li_m(|W_h|^2)W_h}\Big]
=  \dl_{km}( k+1) |\jb{f, h}|^{2k}\jb{f, h}.
\label{Z1}
\end{align}

\noi
Here, $\dl_{km}$ denotes the Kronecker delta function.
\end{lemma}

Besides~\eqref{W1a}, we will use  the following identity:
\begin{align}
u e^\frac{u^2}{2} = \frac{1}{\sqrt{2\pi}} \int_\R x e^{xu - \frac{x^2}{2}} dx. 
\label{Z2}
\end{align}

\noi
This follows from 
differentiating~\eqref{W1a} in $u$.

\begin{proof}[Proof of Lemma~\ref{LEM:Z1}]
Let $G_1$ be as in~\eqref{T1} with $\al = 1$.
Let $-1 < t< 0$.
From~\eqref{W1a} and~\eqref{Z2},  we have 
\begin{align*}
G_1(t,  |W_f|^2) W_f 
 & = \frac{1}{(1-t)^2} \Re W_f  e^{\frac{-t}{1-t}\big((\Re W_f)^2 + (\Im W_f)^2\big) }\notag \\
& \hphantom{XX}
+ \frac{i}{(1-t)^2} \Im W_f  e^{\frac{-t}{1-t}\big((\Re W_f)^2 + (\Im W_f)^2\big) }\notag \\
& = \frac{1}{\sqrt{-2t}(1-t)^\frac{3}{2}}
\frac{1}{2\pi}\int_{\R^2} (x_1+ix_2) e^{-\frac{x_1^2 + x_2^2}{2}}
e^{\sqrt{\frac{-2t}{1-t}}(x_1 \Re W_f + x_2 \Im W_f)} dx_1 dx_2.
\end{align*}

\noi
Given $x_1, x_2, y_1, y_2 \in  \R$, 
let $x = x_1 + i x_2$ and $y = y_1 + i y_2$.
Then, for any $ -1< t, s <0$, 
from Lemma~\ref{LEM:W0}, we have 
\begin{align}
\int_{\O} & G_1(t, W_f(\o)) W_f(\o) \cj{G_1(s, W_h(\o))W_h(\o)} dP(\o)\notag\\
& = 
\frac{1}{\sqrt{-2t}(1-t)^\frac{3}{2}}
\frac{1}{\sqrt{-2s}(1-s)^\frac{3}{2}}
\frac1 {4\pi^2}
\int_{\R^4} x \cj ye^{-\frac{|x|^2 + |y|^2 }{2}}\notag \\
& \hphantom{XXXXXXXX}
\times 
\int_\O
\exp\Big(
\Re W_{
\sqrt{\frac{-2t}{1-t}}\cj x  f
+ \sqrt{\frac{-2s}{1-s}}\cj y h}\Big)
dP dx_1dx_2dy_1dy_2  \notag \\
& = 
\frac{1}{\sqrt{-2t}(1-t)^\frac{3}{2}}
\frac{1}{\sqrt{-2s}(1-s)^\frac{3}{2}}
\frac1 {4\pi^2}
\int_{\R^4} x\cj y  e^{-\frac{|x|^2}{2(1-t)}-\frac{|y|^2}{2(1-s)}}
\notag \\
& \hphantom{XXXXXXXX}
\times 
e^{\frac 12 \Re \big(
\sqrt{\frac{-2t}{1-t}} \sqrt{\frac{-2s}{1-s}}\cj x y \jb{f,  h}\big)}
 dx_1dx_2dy_1dy_2  \notag \\
 \intertext{By a change of variables
 and applying~\eqref{W1a} and~\eqref{Z2}, we have} 
& =  \frac{1}{2 \sqrt{ts}}
\frac1 {4\pi^2}
\int_{\R^4} x\cj y  e^{-\frac{|x|^2}{2}-\frac{|y|^2}{2}}
e^{\sqrt{ts} \Re (\cj x y \jb{f,  h})}
 dx_1dx_2dy_1dy_2  \notag \\
& =  
 \jb{f, h} \frac1 {4\pi}
\int_{\R^2} |y|^2 
e^{ - \frac 1 2(1- ts |\jb{f, h}|^2) |y|^2  }dy_1dy_2 \notag \\
\intertext{By integration by parts, we have }
& =  
 \frac{\jb{f, h}}{1- ts |\jb{f, h}|^2}
\frac1 {2\pi}
\int_{\R^2} 
e^{ - \frac 1 2(1- ts |\jb{f, h}|^2) |y|^2  }dy_1dy_2 \notag \\
& = 
 \frac{\jb{f, h}}{(1- ts |\jb{f, h}|^2)^2}
= \sum_{k = 0}^\infty(k+1)  t^ks^k |\jb{f, h}|^{2k}\jb{f, h}.
\label{Z3}
\end{align}

\noi
Hence, it follows from~\eqref{T1} and~\eqref{Z3} that 
\begin{align*}
\sum_{k = 0}^\infty(k+1)  t^ks^k |\jb{f, h}|^{2k}\jb{f, h}
 = \sum_{k, m  = 0}^\infty 
t^ks^m
\int_\O \Li_k(|W_f(\o)|^2)W_f \cj{\Li_m(|W_h(\o)|^2)W_h} dP(\o).
\end{align*}

\noi
By comparing the coefficients of $t^ks^m$, we obtain~\eqref{Z1}.
\end{proof}

As a preliminary step to the proof of Proposition~\ref{PROP:nonlin}, 
we first estimate the size of the Fourier coefficient of $F_N(u)$.

\begin{lemma}\label{LEM:Z2}
Let $m \geq 2$ be an integer.
Then, for any $\theta > 0$, there exists $C_{m, \theta} > 0$ such that 
\begin{align}
\|  \jb{F_N(u), e_n}_{L^2_x}  \|_{L^2(\mu)}
\leq C_{m, \theta} \frac{1}{(1+ |n|^2)^{\frac 12(1 - \theta)}}
\label{Z4}
\end{align}

\noi
for any $n \in \Z^2$ and any $N \in \N$.	
Moreover, given positive $\eps < \frac 12$
and any $ 0 < \theta \leq 1 - \eps $, there exists $C_{m,\theta,  \eps} > 0$ such that 
\begin{align}
\|\jb{ F_M(u) - F_N(u), e_n}_{L^2_x} \|_{L^2(\mu)}
\leq C_{m, \theta, \eps} \frac{1}{N^\eps(1+ |n|^2)^{\frac{1}{2}(1-\theta - \eps)}}
\label{Z5}
\end{align}

\noi
for any $n \in \Z^2$ and any $ M \geq N \geq 1$.

\end{lemma}

\begin{proof}
We first prove~\eqref{Z4}.
Let $m \geq2$ be an   integer and $N \in \N $.
From~\eqref{T3} with~\eqref{W6}, we have 
\begin{align}
 F_N(u) 
 = (-1)^{m+1} (m-1)! \s_N^{m-\frac 12 } \cdot \P_N \Big\{  
 L_{m-1}^{(1)}\big(\big|W_{ \eta_N(x)}\big|^2\big) \cj{W_{ \eta_N(x)} }\Big\}.
\label{Z6}
\end{align}

\noi
Clearly, $ \jb{F_N   (u),  e_n}_{L^2_x} = 0$
when $|n|> N$.
Thus, we only need to consider the case $|n|\leq N$.
From Lemma~\ref{LEM:Z1} with~\eqref{Z6},~\eqref{W3b} and~\eqref{W4}, 
we have 
\begin{align}
\|  \jb{F_N   (u),  e_n}_{L^2_x}  \|_{L^2(\mu)}^2
&  = 
\big[ (m-1)!\big]^2 \s_N^{2m-1}
\int_{\T^2_x\times\T^2_y}
\cj{e_n(x)} e_n(y)
\notag \\
& \hphantom{XX}
 \times \int_\O 
 L_{m-1}^{(1)}\big(\big|W_{ \eta_N(x)}\big|^2\big)\cj{W_{ \eta_N(x)}}
  L_{m-1}^{(1)}\big(\big|W_{ \eta_N(y)}\big|^2\big)W_{ \eta_N(y)}
dP dx dy\notag \\
 & =  
m!  (m-1)!
\int_{\T^2_x\times\T^2_y}
 |\g_N(x-y)|^{2m-2}\g_N(x-y)
\cj{ e_n(x-y)}
 dx dy \notag \\
&  = C_m 
\F \big[ |\g_N|^{2m-2}\g_N\big] (n).
\label{Z7}
\end{align}

\noi
Let $\G_{2m-1}(n)$ be as in~\eqref{N1a}.
For $(n_1, \dots, n_{2m-1}) \in \G_{2m-1}(n)$, 
we have $ \max_j |n_j| \ges |n|$.
Thus,  we have 
\begin{align}
\F \big[ |\g_N|^{2m-2}\g_N\big] (n)
= \sum_{\substack{\G_{2m-1}(n)\\
|n_j| \leq N}} 
\prod_{j = 1}^{2m-1} \frac{1}{1+ |n_j|^2}
\leq d_{m, \theta}  \frac{1}{(1+|n|^2)^{1-\theta}}.
\label{Z8}
\end{align}

\noi
Hence,~\eqref{Z4} follows from~\eqref{Z7}
and~\eqref{Z8}.

Next, we prove~\eqref{Z5}.
Let $M \geq N \geq 1$.
Proceeding as before with~\eqref{Z6}, Lemma~\ref{LEM:Z1}, and~\eqref{W4}, 
we have
\begin{align}
\|\jb{   F_M(u) &   - F_N(u),  e_n}_{L^2_x} \|_{L^2(\mu)}^2\notag \\
  &  = C_m 
\Big\{
\ind_{[0, M]}(|n|)\F \big[ |\g_M|^{2m-2}\g_M\big] (n)
 - \ind_{[0, N]}(|n|)\F \big[ |\g_N|^{2m-2}\g_N\big] (n)
\Big\}\notag\\
  &  = C_m 
\ind_{[0, N]}(|n|)\Big\{
\F \big[ |\g_M|^{2m-2}\g_M\big] (n)
 - \F \big[ |\g_N|^{2m-2}\g_N\big] (n)
\Big\}\notag\\
  & \hphantom{X}
   + C_m 
\ind_{(N, M]}(|n|)\F \big[ |\g_M|^{2m-2}\g_M\big] (n).
\label{Z9}
\end{align}

\noi
On the one hand, noting that $|n| > N$, 
we can use~\eqref{Z8}
to estimate
the second term on the right-hand side of~\eqref{Z9}, 
yielding~\eqref{Z5}.
On the other hand,  noting that 
\begin{align*}
\Big|
\F \big[ |\g_M|^{2m-2}\g_M\big] (n)
& - \F \big[  |\g_N|^{2m-2}\g_N\big] (n)\Big| \notag\\
& \leq \sum_{\substack{\G_{2m-1}(n)\\
 |n_j| \leq M\\
\max_j|n_j| \geq N}} 
\prod_{j = 1}^{2m-1} \frac{1}{1+|n_j|^2}
\leq d_{m, \theta}  \frac{1}{\max(N^2,  1+|n|^2)^{1-\theta}}, 
\end{align*}

\noi
we can estimate
the first term on the right-hand side of~\eqref{Z9}
by~\eqref{Z5}.
\end{proof}

Next, we use  the Wiener chaos estimate 
(Lemma~\ref{LEM:hyp3}) to extend Lemma~\ref{LEM:Z2}
for any finite $p\geq 1$.

\begin{corollary}\label{COR:Z10}
Let $m \geq 2$ be an   integer.
Then, for any $\theta > 0$, there exists $C_{m, \theta} > 0$ such that 
\begin{align}
\|  \jb{F_N(u), e_n}_{L^2_x}  \|_{L^p(\mu)}
\leq C_{m, \theta}  (p-1)^{m- \frac{1}{2}}\frac{1}{(1+ |n|^2)^{\frac 12 (1-\theta)}}
\label{Z10}
\end{align}

\noi
for any $n \in \Z^2$ and any $N \in \N$.	
Moreover, given positive $\eps < \frac 12$
and any $ 0 < \theta \leq 1 - \eps $, there exists $C_{m, \theta,  \eps} > 0$ such that 
\begin{align}
\|\jb{ F_M(u) - F_N(u), e_n}_{L^2_x} \|_{L^p(\mu)}
\leq C_{m, \theta, \eps} (p-1)^{m- \frac{1}{2}} \frac{1}{N^\eps(1+ |n|^2)^{\frac 12 (1-\theta-\eps)}}
\label{Z11}
\end{align}

\noi
for any $n \in \Z^2$ and any $ M \geq N \geq 1$.

\end{corollary}

\begin{proof}
Let $m \geq 2$ be an even integer.
In view of Lemma~\ref{LEM:Z2}, 
we only consider  the case $p > 2$.	
From~\eqref{T3} with 
\eqref{L1b}, we have 
\begin{align*}
F_N(u) = |u|^{2m - 2}u + \sum_{\l = 0}^{m-1} a_{m, \l, N} |u|^{2\l-2} u.
\end{align*}

\noi
Recalling~\eqref{N1a} and~\eqref{N1b}, we have 
\begin{align}
\jb{F_N(u), e_n}_{L^2_x}
= \sum_{\l = 0}^{m} a_{m, \l, N}
 \sum_{\substack{\G_{2\l-1}(n)\\|n_j| \leq N}}
\prod_{j = 1}^{2\l-1} \frac{g^*_{n_j}}{\sqrt{1+|n_j|^2}}.
\label{Z12}
\end{align}

%

%
%
%

\noi
Noting that the right-hand side of \eqref{Z12}  is a sum of polynomials of degree (at most) $2m-1$ in $\{g_n\}_{n \in \Z^2}$, 
the bound~\eqref{Z10} follows
from Lemma~\ref{LEM:Z2} and~\ref{LEM:hyp3}.
The proof of~\eqref{Z11} is analogous and we omit the details.
\end{proof}

Finally, we  present the proof of Proposition~\ref{PROP:nonlin}.

\begin{proof}[Proof of Proposition~\ref{PROP:nonlin}]
Let $s< 0$.
Choose sufficiently small  $\theta > 0$ such that $s+ \theta < 0$.
Let $ p \geq 2$. Then, it follows from 
Minkowski's integral inequality and 
\eqref{Z10} that 
\begin{align*}
\big\|\|  F_N(u) \|_{H^s}\big\|_{L^p(\mu)}
& \leq \bigg(\sum_{n \in \Z^2} \jb{n}^{2s}
\|\jb{F_N(u), e_n}_{L^2_x} \|_{L^p(\mu)}^2
\bigg)^\frac{1}{2}\\
& \les 
(p-1)^{m- \frac{1}{2}}
\bigg(\sum_{n \in \Z^2} \jb{n}^{-2 +2\theta +2s}
\bigg)^\frac{1}{2}
\leq  C_{m, p} < \infty
\end{align*}

\noi
since $s+\theta< 0$.
Similarly, given $\eps > 0$ such that $s + \eps < 0$, 
choose sufficiently small $\theta > 0$ such that $s+ \theta+\eps < 0$.
Then, from~\eqref{Z11}, we have 
\begin{align*}
\big\|\| F_M(u) - F_N(u) \|_{H^s}\big\|_{L^p(\mu)}
& \leq \bigg(\sum_{n \in \Z^2} \jb{n}^{2s}
\|\jb{ F_M(u) - F_N(u), e_n}_{L^2_x} \|_{L^p(\mu)}^2
\bigg)^\frac{1}{2}\\
& \les
(p-1)^{m- \frac{1}{2}}
 \frac{1}{N^\eps} 
\bigg(\sum_{n \in \Z^2} \jb{n}^{-2 +2\theta +2\eps + 2s}
\bigg)^\frac{1}{2}
\les(p-1)^{m- \frac{1}{2}}
 \frac{1}{N^\eps} 
\end{align*}

\noi
since $s + \theta + \eps < 0$.
This proves~\eqref{nonlin1}.	
\end{proof}

\section{Extension to 2-$d$ manifolds and domains in $\R^2$}
\label{SEC:mfd}

Let    $(\M,g)$ be   a two-dimensional compact Riemannian manifold without boundary or a bounded domain in $\R^2$. 
In this section, we discuss the extensions of Propositions \ref{PROP:WH1}, \ref{PROP:Gibbs1}, 
and \ref{PROP:nonlin} to  $\M$.

Let $\{\varphi_n\}_{n \in \N }$  be an orthonormal basis   of $L^2(\M)$ consisting  of eigenfunctions of $-\Delta_g$
(with the Dirichlet or Neumann boundary condition when $\M$ is a domain in $\R^2$)
with 
the corresponding eigenvalues $\{\lambda_n^2\}_{n \in \N }$, which we assume to be arranged in the increasing order.
Then,  by Weyl's asymptotics, we have  
\begin{align}
\lambda_n\approx n^{\frac 12}.
\label{Weyl1}
\end{align}

\noi
See,  for example, \cite[Chapter~14]{Zworski}.

Let $\{g_n(\omega)\}_{n\in \N}$ be a sequence of independent standard complex-valued Gaussian random variables on a probability space $(\Omega, \mathcal{F},P)$. We define the Gaussian measure $\mu$ 
as the induced probability measure under the map:
\begin{align}
 \o \in \O \longmapsto u(x) = u(x; \o) = \sum_{n \in \N }\frac{g_{n}(\omega)}{(1+\lambda^2_n)^{\frac{1}{2}}}\, \varphi_n(x).
\label{4G5}
 \end{align}

Note that all  the results in Sections \ref{SEC:2} and \ref{SEC:3} still hold true in this general context with exactly the same proofs, except for Lemma~\ref{LEM:WH2} and Lemma~\ref{LEM:Z2}, where we used standard Fourier analysis
on $\T^2$. 
In the following, 
we will instead use classical properties of the spectral functions of the Laplace-Beltrami operator.

\medskip 

 Let us now define the Wick renormalization in this context. 
Let $u$ be as in \eqref{4G5}.
Given $N\in \N$, we define the projector $\P_{N}$ by
\[u_N = \P_N u= \sum_{\lambda_{n}\leq N }\ft u(n) \varphi_{n}.\]

\noi
We also define $\s_N$   by
\begin{equation}
\s_N (x)  = \E [|u_N(x)|^2]
= \sum_{\lambda_{n}\leq N}\frac{|\varphi_n(x)|^2 }{1+\lambda_{n}^{2}} \les \log N,
\label{4sigma}
\end{equation}

\noi
where the last inequality follows from \cite[Proposition 8.1]{BTT1} and 
Weyl's law \eqref{Weyl1}. 
Unlike $\s_N$  defined in \eqref{Wick4}  
for the flat torus $\T^2$, 
the function $\s_N$ defined above depends on $x \in \M$.
Note that $\s_N (x) > 0$ for all $x \in \M$.
The Wick ordered  monomial $:\! |u_N|^{2m}  \!:$ is then defined  by
\begin{equation}
 \,  :\! |u_N|^{2m} \!: \, =
 (-1)^m m! \cdot L_m(|u_N|^2; \s_N).
\label{4L3}
\end{equation}

\noi
By analogy with \eqref{W3} and \eqref{W3a} we define
\begin{align}
\eta_N(x) (\cdot) & := \frac{1}{\s_N^\frac{1}{2}(x)}
\sum_{ \lambda_{n} \leq N} \frac{\cj{\varphi_n(x)}}{\sqrt{1+{\lambda^{2}_{n}}}}\varphi_n(\cdot), 
\label{4eta}\\
\g_N (x,y) & := 
\sum_{  \lambda_{n} \leq N} \frac{\cj{\varphi_n(x)}\varphi_n(y)}{1+ \lambda_{n}^2}, 
\label{defg}
\end{align}

\noi
for $x, y \in \M$. We simply set $\g = \g_\infty$ when $N = \infty$.

From the definition \eqref{4sigma} of $\s_{N}$,  we have $ \| \eta_N(x)\|_{L^2(\M)} = 1$
for all $x \in \M$.
Moreover, 
 we have 
\begin{align}
\jb{\eta_M(x), \eta_N(y)}_{L^2(\M)}
= \frac{1}{\s_M^\frac{1}{2}(x) \s_N^\frac{1}{2}(y)} \g_N(x,y)
\label{4W4}
\end{align}

\noi
for all $x,y\in \M$ and $M\geq N$.

We now introduce the spectral function of the Laplace-Beltrami operator on $\M$ as 
\begin{equation*}
\pi_j(x,y)= \sum_{\lambda_n\in(j-1,j]}  {\cj{\varphi_n(x)}\varphi_n(y)},\qquad 
\end{equation*}

\noi
for $x,y\in \M$ and $j \in \Z_{\geq 0}$.
From~\cite[(1.3) and (1.5)  with $q = \infty$]{SmSo},  we have
 the bound $\pi_j(x,x)\leq C(j+1)$, 
 uniformly in $x\in \M$.
 Therefore, by Cauchy-Schwarz inequality, we obtain
\begin{equation}\label{pi}
\vert \pi_j(x,y)\vert \leq  \sum_{\lambda_n\in(j-1,j]}  \vert \varphi_n(x)\vert \vert \varphi_n(y)\vert     \leq C(j+1), 
\end{equation}

\noi
uniformly in $x,y\in \M$.

Let $\s$ be a weighted counting measure  on $\Z_{\geq 0}$ defined by 
$\s=\sum_{j=0}^{\infty}(j+1)\delta_j$, where $\dl_j$ is the Dirac delta measure at $j \in \Z_{\geq 0}$.
We define  the operator $L$ by 
 \[L: c=\{c_j\}_{j =  0}^\infty  \longmapsto \sum_{j=0}^{\infty}c_j \pi_j.\]

\noi
Then, we have the following boundedness of the operator $L$.

\begin{lemma}\label{op}
Let  $1\leq q\leq 2$. 
Then, the operator $L$ defined above is continuous 
from $\ell^q(\Z_{\geq 0}, \s)$ into $L^{q'}(\M^2)$.
Here, $q'$ denotes the H\"older conjugate of $q$.

\end{lemma}

\begin{proof}
 By interpolation, it is enough to consider the endpoint cases $q=1$ and $q=2$. 
 
  \smallskip
 \noi
  $\bullet$ {\bf Case 1:} $q = 1$.
\quad Assume that $c\in  \ell^1(\Z_{\geq 0},\s)$. Then, from \eqref{pi}, we get
 \begin{equation*}
 \vert L(c)(x,y)\vert 
 \leq  \sum_{j=0}^{\infty}\vert c_j\vert  \vert \pi_j(x,y)\vert 
 \leq C   \sum_{j=0}^{\infty}(j+1)\vert c_j\vert   = \| c\|_{\l^1(\s)}.
 \end{equation*}
 
 \noi
for all  $x,y\in \M$.
This implies the result for $q=1$.
 
 \smallskip
 \noi
  $\bullet$ {\bf Case 2:} $q = 2$.
  \quad 
  Assume that $c\in  \ell^2(\Z_{\geq 0},\s)$. By the orthogonality of the eigenfunctions $\varphi_n$, 
   we have
  \begin{equation}
  \int_{\M}\vert L(c)(x,y)\vert^2 dx =\sum_{j=0}^{\infty}\vert c_j\vert^2 \pi_j(y,y).
\label{pi2}
   \end{equation}
  
  \noi
From \eqref{pi} and \eqref{pi2}, we deduce that 
    \begin{equation*}
  \int_{\M^2}\vert L(c)(x,y)\vert^2 dx dy \leq C \sum_{j=0}^{\infty}(j+1)\vert c_j\vert^2  
  = \| c\|_{\l^2(\s)}^2.
  \end{equation*}

\noi
This implies the result for $q=2$.
 \end{proof}

 Next, we extend the definition of $\g_N$ to 
 general values of $s$:
\begin{equation*}
\g_{s,N} (x,y)  := 
\sum_{  \lambda_{n} \leq N}\frac{\cj{\varphi_n(x)}\varphi_n(y)}{(1+ \lambda^2_{n})^{\frac s2}}
\end{equation*}

\noi
for $x,y\in \M$.
When $N = \infty$, 
we simply set $\g_s = \g_{s, \infty}$ as before.
Note that when $s = 2$, $\g_{2, N}$ 
 and $\g_2$ correspond to $\g_N$ and $\g$
 defined in \eqref{defg}.

\begin{lemma}\label{lemg}
Let $s > 1$.
Then,  the sequence $\{\g_{s, N}\}_{N\in \N}$ converges to $\g_s$ in $L^p(\M^2)$
for all $2\leq p<\frac{2}{2-s}$ when $s \leq 2$ and $2 \leq p \leq \infty$ when $s \geq 2$.
Moreover, for the same range of $p$, 
there exist $C>0$ and $\kappa>0$ such that 
\begin{equation}\label{4cauchy}
\|\g_{s, M}-\g_{s, N}\|_{L^{p}(\M^2)}\leq \frac{C}{N^{\kappa}},
\end{equation}
for all $ M \geq N \geq 1$.
\end{lemma}

\begin{proof}
Given $ M \geq N \geq 1$, define $\al_{N, M}(x, y)$ and 
$\be_{N, M}(x, y)$ by 
\begin{align}
\alpha_{N,M}(x,y)
:\!& =\g_{s, M}(x,y)-\g_{s, N}(x,y) \notag\\
&=\sum_{ N<\lambda_n \leq M} \frac{\cj{\varphi_n(x)}\varphi_n(y)}{(1+ \lambda_{n}^2)^\frac{s}{2} }=\sum_{j=N+1}^{M} \sum_{\lambda_n\in(j-1,j]} \frac{\cj{\varphi_n(x)}\varphi_n(y)}{(1+ \lambda_{n}^2)^\frac{s}{2}}
\label{4A1}
\intertext{and}
\beta_{N,M}(x,y):\! & =\sum_{j=N+1}^{M} \frac{1}{(1+j^2)^\frac{s}{2}}\sum_{\lambda_n\in(j-1,j]}  {\cj{\varphi_n(x)}\varphi_n(y)}=
\sum_{j=N+1}^{M} \frac{\pi_j(x,y)}{(1+j^2)^\frac{s}{2}} .\notag
\end{align}

 Let us first estimate the difference $\alpha_{N,M} - \be_{N,M}$:
\begin{align*}
\vert    \alpha_{N,M}(x,y)-\beta_{N,M}(x,y)   \vert
 &\leq  \sum_{j=N+1}^{M} \sum_{\lambda_n\in(j-1,j]}   
 \bigg\vert  \frac{1}{(1+ \lambda_{n}^2)^\frac{s}{2}}-   \frac{1}{(1+ j^2)^\frac{s}{2}}        \bigg\vert \vert\varphi_n(x)\vert \vert \varphi_n(y)\vert\\
 &\leq C \sum_{j=N+1}^{M} \frac{1}{j^{s+1}}\sum_{\lambda_n\in(j-1,j]}   \vert\varphi_n(x)\vert \vert \varphi_n(y)\vert.
\end{align*}

\noi
Then,  by \eqref{pi},  we obtain
\begin{equation}
\vert    \alpha_{N,M}(x,y)-\beta_{N,M}(x,y)   \vert  \leq \frac{C}{N^{s-1}}.
\label{4A2}
\end{equation}

Next, we estimate $\beta_{N,M}$.
Define a sequence $c = \{ c_j \}_{j = 0}^\infty$
by setting 
\begin{align*}
c_j=
\begin{cases}
\frac{1}{(1+j^2)^\frac{s}{2}},  & \text{if }N+1\leq j \leq M,\\
 0, &  \text{otherwise.}
\end{cases}
\end{align*}

\noi
Note that 
$c\in \ell^q(\N,\s) $ for  $\frac 2s<q\leq 2$. 
Hence, it follows from Lemma\;\ref{op}
that, given any $2\leq p<\frac{2}{2-s}$, there exist $C>0$ and $\kappa > 0$ such that 
\begin{equation}
\Vert  \beta_{N,M}\Vert_{L^p(\M^2)}    =\bigg\Vert \sum_{j=N+1}^{M} \frac{\pi_j}{(1+j^2)^\frac{s}{2}} \bigg\Vert_{L^p(\M^2)} \leq C\bigg( \sum_{j=N+1}^{M} \frac{j+1}{(1+j^2)^{\frac{s}{2}p'}}\bigg)^{\frac{1}{p'}} \leq \frac{C}{N^\kappa}.
\label{4A3}
\end{equation}

\noi
The desired estimate \eqref{4cauchy} follows
from \eqref{4A1}, \eqref{4A2}, and \eqref{4A3}.
\end{proof}

As in the case of the flat torus, 
define $G_N$, $N \in \N$,  by 
\begin{align*}
G_N(u) = \frac 1{2m} \int_{\M}:\! |\P_N u|^{2m}\!: dx.
\end{align*}

\noi
Then,  we have the following extension of  Proposition~\ref{PROP:WH1}

\begin{proposition}\label{4LEM:WH2}
Let $m \geq2$ be an   integer.
Then, $\{G_N(u)\}_{N\in \N}$
is a Cauchy sequence
in $L^p( \mu)$ for any $p\geq 1$.
More precisely, 
there exists $C_{m} > 0$ such that 
\begin{align}
\| G_M(u) - G_N(u) \|_{L^p(\mu)}
\leq C_{m} (p-1)^m\frac{1}{N^\frac{1}{2}}
\label{4WH2}
\end{align}

\noi
for 
 any $p \geq 1$ and 
any $ M \geq N \geq 1$.

\end{proposition}

As in Section \ref{SEC:2}, we make use of the white noise functional on $L^2(\M)$.
Let $w(x; \o)$ be the mean-zero complex-valued Gaussian white noise on $\M$
defined by 
\[ w(x; \o) = \sum_{n \in \N} g_n(\o) \varphi(x).\] 

\noi
We then define the white noise functional $W_{(\cdot)}: L^2(\M) \to L^2(\O)$ by 
\begin{align}
 W_f = \jb{f, w(\o)}_{L^2(\M)} = \sum_{n \in \N} \ft f(n) \cj{g_n(\o)}. 
\label{4white}
 \end{align}

\noi
Note that Lemma \ref{LEM:W0} and hence Lemma \ref{LEM:W1} also hold  on $\M$.

\begin{proof} 
Thanks to the Wiener chaos estimate (Lemma \ref{LEM:hyp3}),  
we are reduced to the case $p=2$.
Given $N \in \N$ and $x \in \T^{2}$, 
it follows from~\eqref{4sigma},~\eqref{4eta}, and~\eqref{4white}
that 
\begin{align}
u_N (x) = \s_N^\frac{1}{2}(x) \frac{u_N(x)}{\s_N^\frac{1}{2}(x) }
= \s_N^{\frac 12}(x)  \cj{W_{ {\eta_N(x)}}}.
\label{4W6}
\end{align}

\noi
Then, from \eqref{4L3} and \eqref{4W6}, we have 
\begin{align}
:\! |u_N|^{2m} \!: \, =
 (-1)^m m!   \s^m_NL_m\bigg(\frac{| u_N|^2}{\s_N}\bigg)
 \, = (-1)^m m! \s_N^m L_m\big(\big|W_{{\eta_N(x)}}\big|^2\big).
 \label{4W7}
\end{align}

\noi
Hence, from~\eqref{4W7}, Lemma~\ref{LEM:W1}, and~\eqref{4W4}, 
we have
\begin{align*}
(2m)^{2} \| &   G_M(u)   - G_N(u)  \|_{L^2(\mu)}^2\\
&   = (m!)^2 
 \int_{\M_x\times \M_y}
\int_{\O} 
\Big[ 
 \s_M^{m}(x) \s_M^{m} (y)L_m\big(\big|W_{{\eta_M(x)}}\big|^2\big)
L_m\big(\big|W_{{\eta_M(y)}}\big|^2\big) \notag \\
& \hphantom{XXXXXXXXXX}
 -  \s_M^{m}(x)\s_N^{m} (y)L_m\big(\big|W_{{\eta_M(x)}}\big|^2\big)
L_m\big(\big|W_{{\eta_N(y)}}\big|^2\big)\notag \\
& \hphantom{XXXXXXXXXX}
 -   \s_N^{m}(x)\s_M^{m} (y)L_m\big(\big|W_{{\eta_N(x)}}\big|^2\big)
L_m\big(\big|W_{{\eta_M(y)}}\big|^2\big) \notag \\
& \hphantom{XXXXXXXXXX}
 +  \s_N^{m}(x)\s_N^{m}(y) L_m\big(\big|W_{{\eta_N(x)}}\big|^2\big)
L_m\big(\big|W_{{\eta_N(y)}}\big|^2\big)
\Big] dP dx dy
 \notag \\
 & = (m!)^2 
 \int_{\M_x\times \M_y}   
 \big[ 
 |\g_M(x,y)|^{2m} - |\g_N(x,y)|^{2m}\big] dx dy.
 \end{align*}
The desired estimate \eqref{4WH2} for $p = 2$
follows from H\"older's inequality and  Lemma\;\ref{lemg}.
\end{proof}

\begin{remark} \rm
Observe that  the  renormalization procedure \eqref{4L3} uses less spectral information than the one used in~\cite[Section~8]{BTT1} for the case $m=2$. Namely, the approach in~\cite{BTT1} needed an explicit expansion of  the spectral function (see~\cite[Proposition~8.7]{BTT1}), but  the inequality \eqref{pi} is enough in the argument above. 

The function $\gamma$ defined in \eqref{defg} is the Green function of the operator $1-\Delta$. 
It is well-known (see for example Aubin~\cite[Theorem~4.17]{Aubin}) that it enjoys the bound
\begin{equation}\label{sing}
|\gamma(x,y)|\leq C\big|\log (d(x,y))\big|, 
\end{equation}

\noi
where $d(x,y)$ is the distance on $\M$ between the points~$x,y\in \M$. 
The bound~\eqref{sing} implies that~$\gamma \in L^{p}(\M^{2})$ for all $1\leq p<\infty$.
However,  we do not know whether~$\g_{N}$ (which is the Green function of a spectral truncation of $1-\Delta$) satisfies a similar bound, uniformly in~$N$. This could have given an alternative proof.  We  refer to~\cite[Remark~8.4]{BTT1} for a discussion on these topics.
\end{remark}

All  the definitions and notations from \eqref{Gibbs1} to \eqref{nonlin2} have  obvious analogues in the general case of the manifold $\M$, and thus we do not redefine them here.

For $N \in \N$, 
let
\begin{align*}
R_N(u) = e^{-G_{N}(u)} = e^{-\frac 1{2m} \int_{\M} :  |u_N|^{2m}  : \, dx}.
\end{align*}

\noi
In view of \eqref{4sigma} and \eqref{4W7}, 
the logarithmic upper bound \eqref{N2} on $-G_N(u)$
also  holds on the manifold $\M$.
Hence, 
by proceeding as in the case of the flat torus, 
we have the following analogue of Proposition~\ref{PROP:Gibbs1}.

\begin{proposition}\label{4PROP:Gibbs1}
Let $m \geq 2$ be an  integer.
Then, 
$R_N(u) \in L^p(\mu)$ for any $p\geq 1$
with a uniform bound in $N$, depending on $p \geq 1$.
Moreover, for any finite $p \geq 1$, 
$R_N(u)$ converges to some $R(u)$ in $L^p(\mu)$ as $N \to \infty$.
\end{proposition}

We conclude this section by 
 the following  analogue of Proposition  \ref{PROP:nonlin}, 
 which enables  us to define the Wick ordered nonlinearity 
 $:\! |u|^{2(m-1)} u \!:$ on the manifold $\M$.

\begin{proposition}
\label{4PROP:nonlin}
Let $m \geq2$ be an   integer and $ s< 0$.
Then, $\{F_N(u)\}_{N\in \N}$ defined in \eqref{nonlin2} and \eqref{T3}
is a Cauchy sequence
in $L^p(\mu;  H^s(\M))$ for any $p\geq 1$.
More precisely, 
 there exist  $\kappa >0 $   and $C_{m, s, \kappa} > 0$ such that 
\begin{align}
\big\|\| F_M(u) - F_N(u) \|_{H^s}\big\|_{L^p(\mu)}
\leq C_{m, s, \kappa} (p-1)^{m-\frac{1}{2}}\frac{1}{N^\kappa}
\label{4nonlin1}
\end{align}

\noi
for 
 any $p \geq 1$ and 
any $ M \geq N \geq 1$.
\end{proposition}

\begin{proof}

Given $N, n \in \N$, define $J_{N, n}$ by 
\begin{equation*}
J_{N,n}=m!  (m-1)!  
\int_{\M_x\times\M_y}
 |\g_{2,N}(x,y)|^{2m-2}\cj{\g_{2,N}(x,y)}\cj{\varphi_{n}(x)}\varphi_{n}(y) dx dy .
\end{equation*}

\noi
 Then, proceeding as in \eqref{Z7} and \eqref{Z9}
with~\eqref{Z6}, Lemma~\ref{LEM:Z1},  and~\eqref{4W4},  we obtain
\begin{align*} 
\|\jb{   F_M(u) &   - F_N(u),  \varphi_n}_{L^2_x} \|_{L^2(\mu)}^2  = 
\ind_{[0, N]}(\lambda_{n})\big(
J_{M,n}-J_{N,n}\big)   + 
\ind_{(N, M]}(\lambda_{n})J_{M,n}
\end{align*}

\noi
for  $M  \geq N \geq 1$.
With $\eps=-s>0$, we then obtain 
\begin{align*}
\big\|\| F_M(u)&  - F_N(u) \|_{H^{-\eps}}\big\|^2_{L^2(\mu)} 
 = \sum_{n\geq 1} \frac1{(1+\lambda_n^{2})^{\eps}} 
\| \jb{ F_M(u)    - F_N(u),  \varphi_n}_{L^2_x}\|^2_{L^2(\mu)} \\
& =\sum_{\lambda_n \leq N} \frac1{(1+\lambda_n^{2})^{\eps}}   (J_{M,n}-J_{N,n})
+\sum_{N<\lambda_n \leq M} \frac1{(1+\lambda_n^{2})^{\eps}}   J_{M,n}       \\
& = C_m \int_{\M_x\times\M_y}
\big( |\g_{2,M}|^{2m-2}\cj{\g_{2,M}} - |\g_{2,N}|^{2m-2}\cj{\g_{2,N}}  \big)\g_{2\eps,N}(x,y) dx dy\\
& \hphantom{X}
+ C_m \int_{\M_x\times\M_y}  |\g_{2,M}|^{2m-2}\cj{\g_{2,M}} \big(\g_{2\eps,M} -\g_{2\eps,N}\big)(x,y)dx dy\\
& =:A_{N,M}+B_{N,M}.
\end{align*}

In the following, 
We only bound the term $B_{N,M}$, since the first term $A_{N,M}$ can be handled similarly. 
Set $\jb{\nb_x}=(1-\Delta_x)^{\frac 12}$.
  Then, noting  that $\jb{\nb_x}^{-1+\eps} \gamma_{2\eps}=\gamma_{1+\eps}$ and that  ${\jb{\nb_x}^{1-\eps} \gamma_{2}=\gamma_{1+\eps}}$, 
it follows from Cauchy-Schwarz  inequality and the fractional Leibniz rule
that 
\begin{align*}
B_{N,M}
& = C_m \int_{\M_x \times \M_y}  \jb{\nb_x}^{1-\eps}\big(   |\g_{2,M}|^{2m-2}\cj{\g_{2,M}}(x,y) \big)\jb{\nb_x}^{-1+\eps}( \g_{2\eps,M} -\g_{2\eps,N}) (x,y)dx dy\\
& = C_m \int_{\M_x \times \M_y}  \jb{\nb_x}^{1-\eps}\big(   |\g_{2,M}|^{2m-2}\cj{\g_{2,M}}(x,y) \big)( \g_{1+\eps,M} -\g_{1+\eps,N}) (x,y)dx dy.\\
&  \leq   C_m \big\|   \jb{\nb_x}^{1-\eps}\big(     |\g_{2,M}|^{2m-2}\cj{\g_{2,M}} \big) \big\|_{L^2(\M^2)}  \|  \g_{1+\eps,M} -\g_{1+\eps,N}   \|_{L^2(\M^2)}  \\
&  \les\big\|    \g_{1+\eps,M}  \big\|_{L^{p_\eps}(\M^2)} \big\|    \g_{2,M}  \big\|^{2m-2}_{L^{q_\eps}(\M^2)}       \|\g_{1+\eps,M} -\g_{1+\eps,N}    \|_{L^2(\M^2)}
\end{align*}

\noi
with $p_{\eps}=\frac2{1-\eps/2}$ and $q_{\eps}=8(m-1)/\eps$. 
Hence,  from Lemma~\ref{lemg} we conclude  that 
\[B_{N,M} \leq \frac{C_{m, \eps}}{N^\kappa}.\]

\noi
By estimating $A_{N, M}$ in an analogous manner, 
we obtain 
\begin{equation}
\big\|\| F_M(u) - F_N(u) \|_{H^{-\eps}}\big\|_{L^2(\mu)}
\leq  \frac{C_{m, \eps}}{N^\kappa}.
\label{XX1}
\end{equation}

\noi
The bound \eqref{4nonlin1} for general $p\geq 2$ follows from~\eqref{XX1} and the Wiener chaos estimate (Lemma~\ref{LEM:hyp3}).
\end{proof}

\section{Proof of Theorem~\ref{THM:1} and Theorem~\ref{THM:2}}\label{Sect5}

In this section, we present the proof of Theorem~\ref{THM:2} 
on a manifold $\M$ (which contains a particular case of the flat torus stated in Theorem~\ref{THM:1}).
Fix an   integer $ m \geq 2$ 
and $s < 0$
in the remaining part of this section.
We divide the proof into three subsections.
In Subsection~\ref{SUBSEC:5.1}, 
we first 
construct global-in-time dynamics for 
the truncated Wick ordered NLS  and 
prove that the corresponding  truncated Gibbs measures
$\Pkn$   are invariant under its  dynamics.
Then, 
we construct a sequence $\{\nu_N\}_{N\in \N}$
of probability measures on space-time functions
such that their marginal distributions at time $t$
are precisely given by 
 the truncated Gibbs measures
$\Pkn$.
In Subsection~\ref{SUBSEC:5.2}, 
we prove a compactness property of $\{\nu_N\}_{N\in \N}$
so that $\nu_N$ converges weakly up to a subsequence.
In Subsection~\ref{SUBSEC:5.3}, 
by Skorokhod's theorem (Lemma~\ref{LEM:Sk}), 
we  upgrade this weak convergence of $\nu_N$
to almost sure convergence of 
new $C(\R; H^s)$-valued random variables, whose laws are given by $\nu_N$,
and complete the proof of Theorem~\ref{THM:2}.

\subsection{Extending the truncated Gibbs measures
onto space-time functions}\label{SUBSEC:5.1} 

Recall that $\P_{N}$ is the spectral projector onto the frequencies $\big\{n \in \N : \lambda_{n}\leq N\big\}$. 
Consider the truncated Wick ordered NLS:
\begin{align}
i \dt u^N + \Dl  u^N = \P_N\big(\! :\!|\P_Nu^N|^{2(m-1)} \P_Nu^N\!: \!\big).
\label{5NLS3}
\end{align}

We first prove global well-posedness of~\eqref{5NLS3}
and invariance of  the truncated Gibbs measure $\Pkn$
defined in ~\eqref{L6}:
\begin{equation*}
d \Pkn = Z_N^{-1} R_N(u) d\mu = Z_N^{-1} e^{-\frac 1{2m} \int_{\M} :| u_N|^{2m} : \, dx} d\mu.
 \end{equation*}

\begin{lemma}\label{LEM:global}
Let  $N \in \N$.
Then, 
the truncated Wick ordered NLS~\eqref{5NLS3}
is globally well-posed in $H^s(\M)$.
Moreover, the truncated Gibbs measure $\Pkn$  
is invariant under the dynamics of 
\eqref{5NLS3}.

\end{lemma}

\begin{proof}

We first prove global well-posedness of 
the truncated Wick ordered NLS~\eqref{5NLS3}.
Given  $N \in \N$,
let  $v^N = \P_N u^N$.
Then,~\eqref{5NLS3}
can be decomposed into 
the nonlinear evolution equation for $v^N$ on the low frequency part $\{\lambda_{n}\leq N\}$:
\begin{align}
i \dt v^N + \Dl  v^N = \P_N\big(\! :\!|v^N|^{2(m-1)} v^N\!: \!\big)
\label{NLS7}
\end{align}

\noi
and a linear ODE for each high frequency $\lambda_{n} > N$:
\begin{align}
i \dt \ft{u^N}(n) =  \lambda_{n}^2 \ft{u^N}(n). 
\label{NLS8}
\end{align}

\noi
As a linear equation, any solution 
$\ft{u^N}(n)$ to~\eqref{NLS8} exists globally in time.
By viewing~\eqref{NLS7} on the Fourier side,
we see that~\eqref{NLS7} 
is a finite dimensional system of   ODEs of dimension $d_{N}=\#\big\{n\;:\: \lambda_{n}\leq N\big\}$, 
where the vector field depends smoothly on 
$\big\{\ft{u^N}(n)\big\}_{\lambda_{n}\leq N}$.
Hence, by the Cauchy-Lipschitz theorem, 
we obtain
local well-posedness of~\eqref{NLS7}.

With~\eqref{T3} we have 
\begin{align*}
\frac{d}{dt}\int_{\M} |v^N|^2 dx 
& = 2\Re \int_{\M} \dt v^N \cj{v^N} dx\\
& =  - 2\Re\bigg( i \int_{\M} |\nb v^N|^2  dx\bigg)\\
& 
\hphantom{X|}
 - 2 (-1)^{m+1} (m-1)! \s_N^{m-1}
\Re \bigg( i 
\int_{\M}   L_{m-1}^{(1)}\Big(\tfrac{|v^N|^2}{\s_N}\Big) |v^N|^2dx \bigg)\\
& = 0.
\end{align*}
	
\noi
In particular, this shows that 
the Euclidean norm  
\[\big\| \{\ft{v^N}(n)\}_{\lambda_{n}\leq N}\big\|_{\C^{d_{N}}}
= \bigg(\sum_{\lambda_{n} \leq N} |\ft{v^N}(n)|^2\bigg)^\frac{1}{2}
= \bigg(\int_{\M} |v^N|^2 dx \bigg)^\frac{1}{2}\]

\noi
is conserved under~\eqref{NLS7}.
This proves global existence for~\eqref{NLS7}
and hence for the truncated Wick ordered NLS~\eqref{5NLS3}.

As in~\eqref{NGibbs1}, write 
$\Pkn = \ft P^{(2m)}_{2, N}\otimes \mu^\perp_N$.
On the one hand,  the Gaussian measure $\mu_N^\perp$ on 
the high frequencies $\{\lambda_{n} > N\}$ is clearly invariant under
the linear flow~\eqref{NLS8}.
On the other hand, 
noting that~\eqref{NLS7} is the finite dimensional Hamiltonian
dynamics corresponding to 
$H^N_\text{\tiny Wick}(v^N)$ with 
\begin{equation*}
H^N_\text{\tiny Wick}(v^N) = \frac{1}{2}\int_{\M} |\nb v^N|^2dx
+ \frac1{2m} \int_{\M} :\! |  v^N|^{2m}\!: dx,
\end{equation*}
we see that 
$\ft P^{(2m)}_{2, N}$ is invariant under~\eqref{NLS7}.
Therefore, 
the truncated Gibbs measure
$\Pkn $ is invariant under
the dynamics of~\eqref{5NLS3}.
\end{proof}

Let $\Phi_N: H^s (\M)\to C(\R; H^s(\M))  $
be the solution map to~\eqref{5NLS3} constructed in Lemma~\ref{LEM:global}.
For $t \in \R$, 
we
use $\Phi_N(t): H^s(\M) \to H^s(\M)$
to denote the map defined by 
$\Phi_N(t)(\phi) = \big(\Phi_N(\phi)\big)(t)$.
We endow $ C(\R; H^s(\M))$ with the compact-open topology.
Namely, we can view  $C(\R; H^s(\M))$
 as a Fr\'echet space 
endowed with the following metric:
\[ d(u, v) = \sum_{j = 1}^\infty \frac{1}{2^j}
\frac{\|u-v\|_{ C([-j, j]; H^s)}}{1+\|u-v\|_{ C([-j, j]; H^s)}}.\]

\noi
Under this topology, a sequence $\{u_n\}_{n \in \N}\subset  C(\R; H^s(\M))$ converges
if and only if it converges uniformly on any compact time interval.
Then, it follows from  the local Lipschitz continuity of $\Phi_N(\cdot)$
that $\Phi_N$ is continuous
from $H^s (\M)$ into $C(\R; H^s(\M))  $.
We now extend $\Pkn$ on $H^s$
to a probability measure $\nu_N$
on $C(\R; H^s(\M))  $ by setting
\[ \nu_N = \Pkn \circ \Phi_N^{-1}.\]

\noi
Namely, $\nu_N$ is the induced probability measure of $\Pkn$
under the map $\Phi_N$.
In particular, we have
\begin{align}
 \int_{C(\R; H^s)}  F(u) d\nu_N (u) = \int_{H^s} F(\Phi_N(\phi)) d \Pkn(\phi)
\label{Y1}
 \end{align}

\noi
for any measurable function $F :C(\R; H^s(\M))\to  \R$.

\subsection{Tightness of the measures $\nu_N$}
\label{SUBSEC:5.2}

In the following, we prove that the sequence
$\{\nu_N\}_{N\in \N}$ of probability measures on $C(\R; H^s(\M))$ is precompact.
Recall the following definition of tightness of a sequence of probability measures.

\begin{definition}\label{DEF:tight}\rm
A sequence $\{ \rho_n\}_{n \in \N}$ of probability measures
on a metric space $\mathcal{S}$ is {\it tight}
if, for every $\eps > 0$, there exists a compact set $K_\eps$
such that $\rho_n(K_\eps^c) \leq \eps$ for all $n \in \N$.
\end{definition}

\noi	
It is well known that 
tightness of a sequence of probability measures
is equivalent to precompactness of the sequence.
See~\cite{Bass}.

\begin{lemma}[Prokhorov's theorem]\label{LEM:Pro}
If a sequence of probability measures 
on a metric space $\mathcal{S}$ is tight, then
there is a subsequence that converges weakly to 
a probability measure on $\mathcal{S}$.
\end{lemma}

The following proposition shows that 
 the family 
 $\{ \nu_N\}_{N \in \N}$ is tight
 and hence,   up to a subsequence,  it  converges weakly to some probability measure $\nu$ on $C(\R; H^s)$.

\begin{proposition}\label{PROP:tight}
Let $s< 0$. Then, the family $\{ \nu_N\}_{N \in \N}$
of the probability measures on 
$C(\R; H^s(\M))$
is tight.
\end{proposition}

\noi
The proof of Proposition~\ref{PROP:tight}
is similar to that of~\cite[Proposition 4.11]{BTT1}.
While~\cite[Proposition 4.11]{BTT1}
proves the tightness of $\{ \nu_N\}_{N \in \N}$
restricted to $[-T, T]$ for each $T>0$, 
we directly prove the tightness of $\{ \nu_N\}_{N \in \N}$ on the whole time interval.

In the following, we first state  several lemmas.
We present the proof of Proposition~\ref{PROP:tight}
at the end of this subsection.
For simplicity of presentation, we use the following notations.
Given $T>0$, we write $L^p_T H^s$ for $L^p([-T, T]; H^s)$.
We use a similar abbreviation for other function spaces in time.
Let $\rho$ be a probability measure on $H^s$.
With a slight abuse of notation, we use $L^p(\rho)H^s$
to denote
\[ \| \phi\|_{L^p(\rho)H^s}
= \big\| \| \phi \|_{H^s}  \big\|_{L^p(\rho)}.\]

The first lemma provides a control on the size
of random space-time functions.
The invariance of $\Pkn$ under the dynamics of~\eqref{5NLS3}
plays an important role.

\begin{lemma}\label{LEM:bound1}
Let $ s< 0$ and $p \geq 1$.
Then, there exists $C_p > 0$ such that 
\begin{align}
\big\| \| u\|_{L^p_T H^s} \big\|_{L^p(\nu_N)} & \leq C_p T^\frac{1}{p}, \label{Y2}\\
\big\| \| u\|_{W^{1, p}_T H^{s-2}} \big\|_{L^p(\nu_N)} & \leq C_pT^\frac{1}{p}, 
\label{Y3}
\end{align}

\noi
uniformly in $N \in \N$.
	
\end{lemma}

\begin{proof}

By Fubini's theorem, the definition ~\eqref{Y1}, the invariance of $\Pkn$ (Lemma~\ref{LEM:global}), 
and H\"older's inequality, we have 
 \begin{align}
\big\| \| u\|_{L^p_T H^s} \big\|_{L^p(\nu_N)}  
& = 
\big\| \| \Phi_N(t) (\phi)  \|_{L^p_T H^s}
\big\|_{L^p(\Pkn)}  
= 
\big\| \| \Phi_N(t) (\phi)  \|_{L^p(\Pkn) H^s}
\big\|_{L^p_T}  \notag \\
& 
= (2T)^\frac{1}{p} 
\| \phi \|_{L^p(\Pkn) H^s}
\leq (2T)^\frac{1}{p}\|R_N \|_{L^{2p}(\mu)} \| \phi \|_{L^{2p}(\mu) H^s}.
\label{Y4}
\end{align}

\noi
Then,~\eqref{Y2}
follows from~\eqref{Y4} with 
 Proposition~\ref{4PROP:Gibbs1}, 
\eqref{4G5}, and Lemma~\ref{LEM:hyp3}.

From~\eqref{5NLS3} and the definition of $F_{N}$, we have
\begin{align}
\big\| \| \dt u\|_{L^p_T H^{s-2}} \big\|_{L^p(\nu_N)} 
\leq
\big\| \| u\|_{L^p_T H^{s}} \big\|_{L^p(\nu_N)} 
+ \big\| \| F_N(u)\|_{L^p_T H^{s-2}} \big\|_{L^p(\nu_N)} .
\label{Y5}
\end{align}

\noi
The first term is estimated by~\eqref{Y2}.
Proceeding as in~\eqref{Y4} with Propositions~\ref{4PROP:Gibbs1} and~\ref{4PROP:nonlin}, we have 
 \begin{align*}
\big\| \| F_N(u)\|_{L^p_T H^{s-2}} \big\|_{L^p(\nu_N)} 
& \leq (2T)^\frac{1}{p}\|R_N \|_{L^{2p}(\mu)} \| F_N(\phi) \|_{L^{2p}(\mu) H^{s-1}}
 \leq C_p T^\frac{1}{p}.
\end{align*}

\noi
This proves~\eqref{Y3}.
\end{proof}

Recall the following lemma 
on deterministic functions from~\cite{BTT1}.	

\begin{lemma}[{\cite[Lemma 3.3]{BTT1}}]\label{LEM:BTT1}
Let $T > 0$ and $1\leq p \leq \infty$.
Suppose that 
$u \in L^p_T H^{s_1}$ and $\dt u \in L^p_T H^{s_2}$
for some $s_2 \leq s_1$.
Then, for $ \dl > p^{-1}(s_1 - s_2)$, we have 
\[ \| u \|_{L^\infty_TH^{s_1 - \dl}} \les \| u \|_{L^p_T H^{s_1}}^{1-\frac 1p}
\| u \|_{W^{1, p}_T H^{s_2}}^{ \frac 1p}.\]

\noi
Moreover, there exist $\al > 0$ and $\theta \in [0, 1] $
such that for all $t_1, t_2 \in [-T, T]$, we have
\[ \| u(t_2) - u(t_1)  \|_{H^{s_1 - 2\dl}} \les |t_2 - t_1|^\al  \| u \|_{L^p_T H^{s_1}}^{1-\theta}
\| u \|_{W^{1, p}_T H^{s_2}}^{ \theta}.\]

\end{lemma}
	
We are now ready to present the proof of Proposition~\ref{PROP:tight}.

\begin{proof}[Proof of Proposition~\ref{PROP:tight}]
Let $s < s_1 < s_2 < 0$.
For $\al \in (0, 1)$, consider the Lipschitz space $C^\al_TH^{s_1} = C^\al([-T, T]; H^{s_1}(\M))$ defined by the norm
\[ \| u \|_{C^\al_T H^{s_1}} = \sup_{\substack{t_1, t_2 \in [-T, T]\\t_1 \ne t_2}}
\frac{\| u(t_1) - u(t_2) \|_{H^{s_1}}}{|t_1 - t_2|^\al} + \|u \|_{L^\infty_T H^{s_1}}.
\]

\noi
It follows from the Arzel\`a-Ascoli theorem
that the embedding 
$C^\al_T H^{s_1} \subset 
C_T H^{s}$ is compact
for each $T>0$.

By Lemma~\ref{LEM:BTT1} with large $p\gg1$ and Young's inequality, we have
\begin{align}
\| u \|_{C^\al_T H^{s_1}}
\les \|u \|_{L^p_TH^{s_2}}^{1-\theta} \|u \|_{W^{1, p}_TH^{s_2 - 2}}^{\theta}
\les \|u \|_{L^p_TH^{s_2}}+ \|u \|_{W^{1, p}_TH^{s_2 - 2}}
\label{Y6}
\end{align}
	
\noi
for some $\al \in (0, 1)$.
Then, it follows from~\eqref{Y6} and 
Lemma~\ref{LEM:bound1} that 
\begin{align}
\big\|\| u \|_{C^\al_T H^{s_1}}\big\|_{L^p(\nu_N)}
\leq C_p T^\frac{1}{p}.
\label{Y7}
\end{align}

For $j \in \N$, let $T_j = 2^j$.
Given $\eps > 0$, define $K_\eps$ by 
\[ K_\eps = \big\{ u \in C(\R; H^s):\, \| u \|_{C^\al_{T_j} H^{s_1}} \leq c_0 \eps^{-1} T_j^{1+ \frac{1}{p}} 
\text{ for all }
j \in \N \big\}.\]

\noi
Then, by Markov's inequality with~\eqref{Y7}
and choosing $c_0 > 0$ sufficiently large, we have 
\[ \nu_N(K_\eps^c) 
\leq c^{-1}_0 C_1 \eps T_j^{-1-\frac{1}p}
\big\|\| u \|_{C^\al_{T_j} H^{s_1}}\big\|_{L^p(\nu_N)}
\leq c^{-1}_0 C_p \eps \sum_{j = 1}^\infty T_j^{-1}
= c^{-1}_0 C_p \eps < \eps.\]

Hence, it remains to prove that $K_\eps$ is compact in $C(\R; H^s)$
endowed with the compact-open topology.
Let $\{u_n \}_{n \in \N} \subset K_\eps$.
By the definition of $K_\eps$, 
$\{u_n \}_{n \in \N}$ is bounded in 
$C^\al_{T_j} H^{s_1}$ for each $j \in \N$.
Then, by a diagonal argument, 
we can extract a subsequence $\{u_{n_\l} \}_{\l \in \N}$
convergent in $C^\al_{T_j} H^s$ for each $j \in \N$.
In particular,  $\{u_{n_\l} \}_{\l \in \N}$ converges uniformly in $H^s$
on any compact time interval.
Hence,  $\{u_{n_\l} \}_{\l \in \N}$ converges in $C(\R; H^s)$
endowed with the compact-open topology.
This proves that $K_\eps$ is compact in $C(\R; H^s)$.
\end{proof}

\subsection{Proof of Theorem~\ref{THM:2}}
\label{SUBSEC:5.3}

It follows from  Proposition~\ref{PROP:tight} and Lemma~\ref{LEM:Pro}
that, passing to a subsequence,  
$\nu_{N_j}$ converges weakly to some probability measure $\nu $
on $C(\R; H^s(\M))$ for any $s< 0$.
The following Skorokhod's theorem tells us that, 
by introducing a new probability space $(\wt \O, \F, \wt P)$
and a sequence of new random variables $\wt {u^N}$ with the same distribution $\nu_N$, 
we can  upgrade
this weak convergence to almost sure convergence of $\wt {u^N}$.
See~\cite{Bass}.

\begin{lemma}[Skorokhod's theorem]\label{LEM:Sk}
Let $\mathcal{S}$ be a complete separable metric space. 
Suppose that $\rho_n$ are probability measures on $\mathcal{S}$
converging weakly to a probability measure $\rho$.
Then, there exist random variables $X_n:\wt \O \to \mathcal{S}$
with laws $\rho_n$
and a random variable  $X:\wt \O \to \mathcal{S}$ with law $\rho$
such that $X_n \to X$ almost surely.

\end{lemma}

By Lemma~\ref{LEM:Sk}, there exist another probability space $(\wt \O, \wt \F, \wt P)$,
a sequence $\big\{ \wt {u^{N_j}}\big\}_{j \in \N}$ of $C(\R; H^s)$-valued random variables, 
and 
a $C(\R; H^s)$-valued random variable $u $
such that 
\begin{align}
\L\big(\wt {u^{N_j}}\big) = \L( u^{N_j}) = \nu_N, 
\qquad \L(u) = \nu, \label{Y8}
\end{align}
	
\noi
and $\wt {u^{N_j}}$ converges to $u$ in $C(\R; H^s)$ almost surely
with respect to $\wt P$.

Next, we determine the distributions of these random variables at a given time $t$.
By Lemma~\ref{LEM:global}, 
we have 
\begin{align}
 \L( u^{N_j}(t)) = P^{(2m)}_{2, N_j}
 \label{Y9}
\end{align} 

\noi
for each $t \in \R$.

\begin{lemma}\label{LEM:end1}
Let $\wt u_{N_j}$ and $u$ be as above.
Then, we have
\begin{align*}
\L\big(\wt {u^{N_j}}(t)\big)  = P^{(2m)}_{2, N_j}
\quad \text{and} \quad
\L(u(t)) = \Pk 
\end{align*}

\noi
for any $t \in \R$.
\end{lemma}

\begin{proof}
Fix $t \in \R$.
Let $R_t:C(\R; H^s) \to H^s$ be the evaluation map defined by $R_t(v) = v(t)$. 
Note that $R_t$ is continuous.
From~\eqref{Y9}, we have
\begin{align}
P^{(2m)}_{2, N_j} = \nu_{N_j}\circ R_t^{-1}.
\label{Y10}
\end{align}

\noi
Denoting by $\nu_{N_j}^t$  the distribution of  $\wt {u^{N_j}}(t)$, 
it follows from~\eqref{Y8} and~\eqref{Y10} that 
\begin{align}
\nu_{N_j}^t = \nu_{N_j}\circ R_t^{-1} = P^{(2m)}_{2, N_j} .
\label{Y11}
\end{align}

Since $\wt {u^{N_j}}$ converges to $u$ in $C(\R; H^s)$ almost surely
with respect to $\wt P$, 
$\wt {u^{N_j}}(t)$ converges to $u(t)$ in $H^s$ almost surely.
Then, 
denoting by $\nu^t$  the distribution of  $u(t)$, 
it follows from 
 the dominated convergence theorem with~\eqref{Y11} that 
\begin{align}
\nu^t(A) = \int \ind_{\{u(t)(\o) \in A\}} d\wt P
= \lim_{j\to \infty} \int \ind_{\big\{\wt {u^{N_j}} (t)(\o) \in A\big\}} d\wt P
= \lim_{j \to \infty} P^{(2m)}_{2, N_j} (A).
\label{Y12}
\end{align}

\noi
Therefore, from~\eqref{Y12} and 	Proposition~\ref{4PROP:Gibbs1}, 
we conclude that 
$\L(u(t)) = \Pk $.
\end{proof}

 Finally, we show that the random variable $u$ is indeed a global-in-time distributional solution to the Wick ordered NLS
\begin{align}
i \dt u + \Dl  u = \, :\!|u|^{2(m-1)} u\!: \;, 
\qquad  (t,x) \in \R \times \M.
\label{5NLS2}
\end{align}
Then, Theorem~\ref{THM:2} follows from Lemmas~\ref{LEM:end1} and~\ref{LEM:end2}.

\begin{lemma}\label{LEM:end2}
Let $\wt {u^{N_j}}$ and $u$ be as above.
Then, 
$\wt {u^{N_j}}$ and $u$
are global-in-time distributional 
solutions to 
the truncated Wick ordered NLS~\eqref{5NLS3} for each $j \in \N$
and
to the Wick ordered NLS~\eqref{5NLS2}, respectively.
\end{lemma}

\begin{proof}

For $j \in \N$, define the 
$\mathcal D'_{t, x}$-valued
random variable $X_j$ by 
\begin{align*}
X_j = i \dt u^{N_j} + \Dl  u^{N_j} - \P_{N_j}\big(\! :\!|\P_{N_j}u^{N_j}|^{2(m-1)} \P_{N_j}u^{N_j}\!: \!\big).
\end{align*}

\noi
Here, 
$\mathcal D'_{t, x}= \mathcal{D}'(\R\times \M)$ denotes the space of space-time distributions
on $\R\times \M$.
We define $\wt X_j$ 
for $\wt {u^{N_j}}$ in an analogous manner.
Since $u^{N_j}$ is a solution to~\eqref{5NLS3}, 
we see that $\L_{\mathcal D'_{t, x}}(X_j) = \dl_0$, 
where $\dl_0$ denotes the Dirac delta measure.
By~\eqref{Y8}, we also have 
\[\L_{\mathcal D'_{t, x}}(\wt X_j) = \dl_0,\]

\noi
for each $j \in \N$.
In particular, 
 $\wt {u^{N_j}}$ is a global-in-time distributional solution
to the truncated Wick ordered NLS~\eqref{5NLS3}
for each $j \in \N$, i.e.
\begin{align*}
i \dt \wt{u^{N_j}} + \Dl  \wt {u^{N_j}} =  \P_{N_j}\big(\! :\!|\P_{N_j} \wt{u^{N_j}}|^{2(m-1)} \P_{N_j}\wt {u^{N_j}}\!: \!\big)
\end{align*}

\noi
in the distributional sense, 
almost surely with respect to $\wt P$.

In view of the almost sure convergence
of  $\wt {u^{N_j}}$  to $u$ in $C(\R; H^s)$, 
we have 
\begin{align*}
i \dt \wt{u^{N_j}} + \Dl  \wt {u^{N_j}} 
\ 
\longrightarrow 
\ i \dt u + \Dl u 
\end{align*}

\noi
in $\mathcal D'(\R\times \M)$ as $j \to \infty$, 
 almost surely
with respect to $\wt P$.
Next, we show the almost sure convergence of $F_{N_j}\big(\wt{u^{N_j}}\big)$
to $F(u) =  \ :\!|u|^{2(m-1)} u\!: $.
For simplicity of notation, let $F_j = F_{N_j}$
and $u_j =  \wt{u^{N_j}}$.
Given $M \in \N$, write
\begin{align}
F_j (u_j) - F(u) 
& = \big( F_j (u_j) - F(u_j) \big)
+  \big( F (u_j) - F_M(u_j) \big)\notag \\
& \hphantom{XXX}
+  \big( F_M (u_j) - F_M(u) \big)
+  \big( F_M (u) - F(u) \big).
\label{Y13}
\end{align}
	
\noi
Then, 
for each fixed $M \geq 1$, 
it follows from 
 the almost sure convergence
of  $\wt {u^{N_j}}$  to $u$ in $C(\R; H^s)$
and the continuity of $F_M$
that the third term on the right-hand side of~\eqref{Y13}
converges to 0 in  $C(\R; H^s)$ as $j \to \infty$, almost surely with respect to $\wt P$.

Fix $T>0$ and let $s < -1$.
Arguing as in~\eqref{Y4}
with Proposition~\ref{4PROP:nonlin}, 
we have
\begin{align}
\big\|\|F (u_j) - F_M(u_j)\|_{L^2_T H^s}\big\|_{L^2(\nu_{N_j})}
& = 
\big\|\|F (\Phi_{N_j} \phi ) - F_M(\Phi_{N_j} \phi)\|_{L^2(P^{(2m)}_{2, N_j}) H^s}\big\|_{L^2_T} \notag \\
& = (2T)^\frac{1}{2}
\|F ( \phi ) - F_M( \phi)\|_{L^2(P^{(2m)}_{2, N_j}) H^s} \notag \\
& \les T^\frac{1}{2} \|R_{N_j}\|_{L^4(\mu)}
\|F ( \phi ) - F_M( \phi)\|_{L^4(\mu) H^s} \notag \\
& \leq CT^{\frac{1}{2}} M^{-\eps},
\label{Y14}
\end{align}

\noi
for some small $\eps > 0$, uniformly in $j \in \N$.	
In the third step, we used the fact that 
$Z_N \ges 1$
in view of 
Proposition \ref{PROP:Gibbs1}: 
$Z_N = \| R_N(u) \|_{L^1(\rho)}\ \to 
 \| R(u) \|_{L^1(\rho)}>0$ as $N \to \infty$.
The fourth term on the right-hand side of~\eqref{Y13}
can be treated in an analogous manner.
Proceeding as in~\eqref{Y14}, we obtain
\begin{align*}
\big\|\|F_j (u_j) - F(u_j)\|_{L^2_T H^s}\big\|_{L^2(\nu_{N_j})}
& \leq 
(2T)^\frac{1}{2} \|R_{N_j}\|_{L^4(\mu)}
\|F_j ( \phi ) - F( \phi)\|_{L^4(\mu) H^s} \notag \\
& \leq CT^{\frac{1}{2}} N_j^{-\eps}.
\end{align*}

\noi
Putting everything together, 
we conclude that, after passing to a subsequence, 
 $F_j(u_j)$ converges to $F(u)$
in $L^2([-T, T]; H^s)$ almost surely with respect to $\wt P$.
Since the choice of $T>0$ was arbitrary, 
we can apply the previous argument iteratively for $T_\l = 2^\l$, $\l\in \N$.
Thus, for each $\l \geq 2$, we obtain  a set $\O_{\l} \subset \O_{\l-1} $ of full measure
such that 
a subsequence $F_{j^{(\l)}}(u_{j^{(\l)}})(\o) $ 
of $F_{j^{(\l-1)}}(u_{j^{(\l-1)}}) $ from the previous step
converges to $F(u)(\o)$
in $L^2([-T_\l, T_\l]; H^s)$ for all $\o \in \O_\l$.
Then, by a diagonal argument, 
passing to a subsequence, 
$F_j(u_j) $ converges to $F(u)$
in $L^2_\text{loc} H^s$ 
almost surely with respect to $\wt P$.
In particular, up to a subsequence, 
$F_j(u_j) $ converges to $F(u)$
in $\mathcal D'(\R\times \M)$
almost surely with respect to $\wt P$.
Therefore, $u$ is a global-in-time distributional solution to~\eqref{5NLS2}.
\end{proof}

\appendix

\section{Example of a concrete combinatorial argument: the case $\M=\T^2$ and $m=3$}
\label{SEC:A}

In this appendix, we present a concrete combinatorial computation on the Fourier side
for the proof of Proposition~\ref{PROP:WH1} when $m= 3$.
The aim of this appendix is to convince readers
of increasing combinatorial complexity
in $m$ .
Compare the $m=3$ case presented here with  the $m=2$ case
in~\cite{BO96}.
This shows that the use of the white noise functional is essential
in establishing our result for general  $m \geq 2$.

Let $G_N(u)$ be as in~\eqref{L4}.
For simplicity, we show that 
$G_N(u)$ is uniformly bounded in~$L^2(\mu)$.
Namely, we prove
\begin{align}
\|G_N(u) \|_{L^2(\mu)} \leq C < \infty
\label{A1}
\end{align}

\noi
independently of $N \in \N$.
Then, a small modification yields 
Proposition~\ref{PROP:WH1} for $p = 2$.
The general case follows from the $p = 2$ case and the Wiener chaos estimate (Lemma~\ref{LEM:hyp3}).

From~\eqref{L1a},~\eqref{L2},~\eqref{L3},  and~\eqref{L4} with \eqref{G5}, we have
\begin{align}
6 G_N(u) & = \int_{\T^2} :\! |u_N|^6 \!:  dx
 = \int_{\T^2} |u_N|^6 - 9 \s_N |u_N|^4 +18 \s_N^2 |u_N|^2 - 6 \s_N^3 dx \notag \\
& = 
\sum_{\substack{\G_{6}(0)\\|n_j| \leq N}}\prod_{j = 1}^{6} \frac{g^*_{n_j}}{\sqrt{1+|n_j|^2}}
- 9 \bigg(  \sum_{|n|\leq N}\frac{1}{1+|n|^2}  \bigg)
\bigg(\sum_{\substack{\G_{4}(0)\\|n_j| \leq N}}\prod_{j = 1}^{4} \frac{g^*_{n_j}}{\sqrt{1+|n_j|^2}}\bigg) \notag \\
& \hphantom{XX}
+18 \bigg(  \sum_{|n|\leq N}\frac{1}{1+|n|^2}  \bigg)^2
\bigg(  \sum_{|n|\leq N}\frac{|g_n|^2}{1+|n|^2}  \bigg)
-6 \bigg(  \sum_{|n|\leq N}\frac{1}{1+|n|^2}  \bigg)^3\notag \\
& =: \I + \II + \III+ \IV,
\label{A2}
\end{align}

\noi
where $\s_N$ is as in~\eqref{Wick4}
and $\G_k(0)$ and $g_n^*$ are as in~\eqref{N1a} and~\eqref{N1b}, respectively.

The basic idea is to regroup the terms in~\eqref{A2}
by  introducing some factorizations, 
and separately estimate each contribution.
Given $\l \in 2\N$, we say that we have a {\it pair} in
$\cj  n = (n_1, \dots, n_\l) \in \G_{\l} (0)$
if $n_j = n_{j'}$
for some odd $j$ and even $j'$.

Let us first consider $\I$.
Given $\cj n \in \G_6(0)$,
there are three cases:
(i) no pair, (ii) 1 pair, and (iii) 3 pairs.
Thus, write $\I$ as 
\[ \I = \I_1 + \I_2 + \I_3, \]

\noi
corresponding to the three cases: (i) no pair, (ii) 1 pair, and (iii) 3 pairs, respectively.
For simplicity of notation, we may drop the frequency restriction $|n| \leq N$ in the following
but it is understood that all the summations are over $\{|n|\leq N\}$.

\medskip

\noi
$\bullet$ {\bf Case 1:} No pair.
\quad In this case, we can easily estimate the contribution from $\I_1$ by 
\begin{align}
\| \I_1\|_{L^2(\mu)}
\les \bigg(\sum_{\G_{6}(0)}\prod_{j = 1}^{6} \frac{1}{1+|n_j|^2}\bigg)^\frac{1}{2}
\leq C < \infty.
\label{A3}
\end{align}

\noi
$\bullet$ {\bf Case 2:} 1 pair.
\quad In this case, there are 9 possibilities 
to form a pair from each of $\{n_1, n_3, n_5\}$
and $\{n_2, n_4, n_6\}$.
Thus, we have
\begin{align*}
\I_2 = 
9 \bigg(  \sum\frac{|g_n|^2}{1+|n|^2}  \bigg)
\bigg(\sum_{\substack{\G_{4}(0)\\n_1 \ne n_2, n_4}}\prod_{j = 1}^{4} \frac{g^*_{n_j}}{\sqrt{1+|n_j|^2}}\bigg).
\end{align*}

\noi
Combining this with $\II$, we have 
\begin{align}
\I_2 +\II & = 
9 \bigg(  \sum\frac{|g_n|^2 - 1}{1+|n|^2}  \bigg)
\bigg(\sum_{\substack{\G_{4}(0)\\n_1 \ne n_2, n_4}}\prod_{j = 1}^{4} \frac{g^*_{n_j}}{\sqrt{1+|n_j|^2}}\bigg) \notag \\
& \hphantom{XX}
- 18 \bigg(  \sum\frac{1}{1+|n|^2}  \bigg)
\bigg(  \sum\frac{|g_n|^2}{1+|n|^2}  \bigg)^2\notag \\
& \hphantom{XX}
 +9 \bigg(  \sum\frac{1}{1+|n|^2}  \bigg)
\bigg(  \sum\frac{|g_n|^4}{(1+|n|^2)^2}  \bigg)\notag \\
& =: \II_1+\II_2 + \II_3.
\label{A4}
\end{align}

Note that $\E[|g_n|^2 - 1] = 0$.
Then, by Lemma~\ref{LEM:hyp3}, we have 
\begin{align}
\|\II_1\|_{L^2(\mu) }
& \les  \bigg\|  \sum\frac{|g_n|^2 - 1}{1+|n|^2}  \bigg\|_{L^4(\mu)}
\bigg\| \sum_{\substack{\G_{4}(0)\\n_1 \ne n_2, n_4}}\prod_{j = 1}^{4} \frac{g^*_{n_j}}{\sqrt{1+|n_j|^2}}
\bigg\|_{L^4(\mu)} \notag \\
& \les  \bigg(  \sum\frac{1}{(1+|n|^2)^2}  \bigg)^\frac{1}{2}
\bigg( \sum_{\G_{4}(0)}\prod_{j = 1}^{4} \frac{1}{1+|n_j|^2}\bigg)^\frac{1}{2}
\leq C < \infty.
\label{A5}
\end{align}

\noi
The terms $\II_2$ and $\II_3$ are treated with other terms in the following.

\medskip

\noi
$\bullet$ {\bf Case 3:} 3 pairs.
\quad In this case, there are 3 scenarios
on the values of $n_1, n_3$, and $n_5$:
(i) $n_1 = n_3 = n_5$,  
(ii) $n_1 = n_3\ne n_5$ up to permutations, 
(iii) all distinct.
Write $\I_3 = \I_{31} + \I_{32} + \I_{33}$, corresponding
to these three cases.

\medskip
\noi
$\circ$ Subcase 3 (i): 
  $n_1 = n_3 = n_5$.
\quad In this case, 
the contribution can be estimated by 
\begin{align}
\|\I_{31}\|_{L^2(\mu) }
& \leq  \bigg\|  \sum\frac{|g_n|^6 }{(1+|n|^2)^3}  \bigg\|_{L^2(\mu)}
\les \bigg(  \sum\frac{1 }{(1+|n|^2)^6}  \bigg)^\frac{1}{2}
\leq C < \infty.
\label{A6}
\end{align}

\noi
$\circ$ Subcase 3 (ii): 
$n_1 = n_3\ne n_5$ up to permutations.
\quad
In this case, we have 
\begin{align}
\I_{32} 
& = \begin{pmatrix}3\\2\end{pmatrix}
\begin{pmatrix}3\\2\end{pmatrix}
 \bigg(  \sum\frac{|g_n|^4 }{(1+|n|^2)^2}  \bigg)
 \bigg(  \sum_{m \ne n} \frac{|g_m|^2 }{1+|m|^2}  \bigg) \notag \\
& = 
9  \bigg(  \sum\frac{|g_n|^4 }{(1+|n|^2)^2}  \bigg)
 \bigg(  \sum \frac{|g_m|^2 }{1+|m|^2}  \bigg)
-9 
 \bigg(  \sum\frac{|g_n|^6 }{(1+|n|^2)^3}  \bigg)\notag\\
& =: \I_{321} + O_{L^2(\mu) }(1).
\label{A7}
\end{align}

\noi
Here, we estimated the second term as in~\eqref{A6}.

\medskip
\noi
$\circ$ Subcase 3 (iii): all distinct.
\quad 
In this case, we have 
\begin{align}
\I_{33} 
& = 6 
 \bigg(  \sum \frac{|g_{n_1}|^2 }{1+|n_1|^2}  \bigg)
\bigg(  \sum_{n_3 \ne n_1} \frac{|g_{n_3}|^2 }{1+|n_3|^2}  \bigg)
\bigg(  \sum_{n_5 \ne n_1, n_3} \frac{|g_{n_5}|^2 }{1+|n_5|^2}  \bigg)\notag\\
& = 6 
 \bigg(  \sum \frac{|g_{n_1}|^2 }{1+|n_1|^2}  \bigg)
\bigg(  \sum_{n_3 \ne n_1} \frac{|g_{n_3}|^2 }{1+|n_3|^2}  \bigg)
\bigg(  \sum_{n_5 \ne n_1} \frac{|g_{n_5}|^2 }{1+|n_5|^2}  \bigg)\notag\\
& \hphantom{XX}
-6  \bigg(  \sum\frac{|g_{n_1}|^2 }{1+|n_1|^2}  \bigg)
 \bigg(  \sum_{n_3 \ne n_1} \frac{|g_{n_3}|^4 }{(1+|n_3|^2)^2}  \bigg) \notag \\
& = 6 
 \bigg(  \sum \frac{|g_{n_1}|^2 }{1+|n_1|^2}  \bigg)
\bigg(  \sum_{n_3 \ne n_1} \frac{|g_{n_3}|^2 }{1+|n_3|^2}  \bigg)
\bigg(  \sum \frac{|g_{n_5}|^2 }{1+|n_5|^2}  \bigg)\notag\\
& \hphantom{XX}
-6  \bigg(  \sum\frac{|g_{n_1}|^2 }{1+|n_1|^2}  \bigg)
 \bigg(  \sum_{n_3 \ne n_1} \frac{|g_{n_3}|^4 }{(1+|n_3|^2)^2}  \bigg) \notag \\
& \hphantom{XX}
-6  \bigg(  \sum\frac{|g_{n_1}|^4 }{(1+|n_1|^2)^2}  \bigg)
 \bigg(  \sum_{n_3 \ne n_1} \frac{|g_{n_3}|^2 }{1+|n_3|^2}  \bigg) \notag \\
& = 6 
 \bigg(  \sum \frac{|g_{n_1}|^2 }{1+|n_1|^2}  \bigg)^3
 -18  \bigg(  \sum\frac{|g_{n_1}|^2 }{1+|n_1|^2}  \bigg)
 \bigg(  \sum \frac{|g_{n_3}|^4 }{(1+|n_3|^2)^2}  \bigg) \notag \\
& \hphantom{XX}
 +12   \bigg(  \sum\frac{|g_{n_1}|^6 }{(1+|n_1|^2)^3}  \bigg)\notag\\
& =: \I_{331}+\I_{332} + O_{L^2(\mu) }(1).
\label{A8}
 \end{align}

From 
\eqref{A4},~\eqref{A7}, and~\eqref{A8}, we have
\begin{align*}
\II_3 +  \I_{321}+ \I_{332}
=  9 \bigg(  \sum\frac{1-|g_n|^2}{1+|n|^2}  \bigg)
\bigg(  \sum\frac{|g_n|^4}{(1+|n|^2)^2}  \bigg).
\end{align*}

\noi
Proceeding as in~\eqref{A5}, we obtain
\begin{align}
\|\II_3 +  \I_{321}+ \I_{332}\|_{L^2(\mu) }\leq C < \infty.
\label{A9}
\end{align}
	
\noi
From~\eqref{A2},~\eqref{A4}, and~\eqref{A8}, we have 
\begin{align*}
\III + \IV + \II_2 +  \I_{331}
=  6 \bigg(  \sum\frac{|g_n|^2-1}{1+|n|^2}  \bigg)^3.
\end{align*}

\noi
Proceeding as in~\eqref{A5}, we obtain
\begin{align}
\|\III + \IV + \II_2 +  \I_{331}
\|_{L^2(\mu) }
& \les  \bigg\|  \sum\frac{|g_n|^2 - 1}{1+|n|^2}  \bigg\|_{L^6(\mu)}^3
\leq C < \infty.
\label{A10}
\end{align}

Finally, putting~\eqref{A2}-\eqref{A10} together, 
we obtain~\eqref{A1}.

\begin{remark}\rm
The above computation merely handles 
the nonlinear part $G_N(u)$ in the truncated Wick ordered Hamiltonian.
In order to prove Theorem~\ref{THM:1}, 
one still needs to estimate $F_N(u)$ in~\eqref{nonlin2},
which has a different combinatorial structure.
For our problem, it is much more efficient to work
on the physical side, using the white noise functional
and the (generalized) Laguerre polynomials.

\end{remark}

\begin{ackno}\rm
T.O.~was supported by the European Research Council (grant no.~637995 ``ProbDynDispEq'').
L.T.~was supported by the grant ``ANA\'E'' ANR-13-BS01-0010-03. 
The authors would like to thank  Martin Hairer
for helpful discussions.
They are also grateful to the anonymous referees
for their comments.

\end{ackno}

\end{document}